\documentclass[preprint,a4]{amsart}%
\usepackage{amsthm,amssymb,amsmath,color}
\usepackage{amsfonts}
\usepackage{amsmath}
\usepackage{amssymb}
\usepackage{graphicx}
\usepackage[latin1]{inputenc}%
\providecommand{\U}[1]{\protect\rule{.1in}{.1in}}

\providecommand{\U}[1]{\protect\rule{.1in}{.1in}}
\newtheorem{theorem}{Theorem}

\newtheorem{corollary}[theorem]{Corollary}

\newtheorem{definition}[theorem]{Definition}

\newtheorem{lemma}[theorem]{Lemma}

\newtheorem{proposition}[theorem]{Proposition}
\newtheorem{remark}[theorem]{Remark}

\def\1{\ensuremath{\mathrm{1}\hspace{-.35em} \mathrm{1}}}

\def\E{\mathbb{E}}

\def\P{\mathbb{P}}

\begin{document}
\title{Regular variation of a random length sequence of random variables and application to risk assessment}
\author[Charles Tillier and Olivier Wintenberger]{C. Tillier and O. Wintenberger}
\address{C. Tillier\\
 \\
O. Wintenberger\\
Sorbonne Universit\'es\\
UPMC Univ Paris 06\\
LSTA, Case 158 4 place Jussieu\\
75005 Paris\\
France \& Department of Mathematical Sciences, University of Copenhagen}
\email{olivier.wintenberger@upmc.fr}

\begin{abstract}
When assessing risks on a finite-time horizon, the problem can often be reduced to the study of a random sequence $C(N)=(C_1,\ldots,C_N)$ of random length $N$, where $C(N)$ comes from the product of a matrix $A(N)$ of random size $N \times N$ and a random sequence $X(N)$ of random length $N$. Our aim is to build a regular variation framework for such random sequences of random length, to study their spectral properties and, subsequently, to develop risk
measures. In several applications, many risk indicators can be expressed from the asymptotic behavior of $\vert \vert C(N)\vert\vert$, for some norm $\Vert \cdot \Vert$. We propose a generalization of Breiman Lemma that gives way to an asymptotic equivalent to $\Vert C(N) \Vert$ and provides risk indicators such as the ruin probability and the tail index for Shot Noise Processes on a finite-time horizon. Lastly, we apply our final result to a model used in dietary risk assessment and in non-life insurance mathematics to illustrate the applicability of our method.

\end{abstract}
\maketitle



\section{Introduction}

Risk analyses play a leading role within fields such as
dietary risk, hydrology, nuclear security, finance and insurance. Moreover, risk analysis is present in the
applications of various probability tools and statistical methods. We see a significant impact on the scientific literature and on public institutions (see \cite{Bert_2010},
\cite{Asmu_2003}, \cite{Embr_1997} or \cite{grandell:1991}).

In insurance, risk theory is useful for getting
information regarding the amount of aggregate claims that an insurance
company is faced with. By doing so, one may implement solvency measures to avoid bankruptcy; see \cite{Mikosch} for a survey of
non-life insurance mathematics. To further illustrate the importance of risk analysis, we turn to the field of dietary risk. Here, toxicologists
determine contamination levels which could later be used by mathematicians to build risk indicators. For example, in \cite{Bert_2008}, authors proposed a dynamic dietary risk model, which takes into account both the elimination and the accumulation of a contaminant in the human body; see \cite{FAO_2003}, \cite{Giba_1982},  and \cite{Bert_2010} for a survey on
dietary risk assessment.

Besides, risk theory typically deals with the occurrences of rare events which are 
functions of heavy-tailed random variables, for example, sums or products  of regularly varying random variables; see \cite{Jessen and Mikosch} and \cite{Mikosch (1999)} for an exhaustive survey on regular variation theory. Non-life insurance mathematics and dietary risk management both deal with a particular kind of Shot Noise Processes $\{S(t)\}_{t>0}$, defined as 
\begin{equation}
S(t)=\sum\limits_{n=0}^{N(t)}X_ih(t,T_i), \quad t>0,\label{ShotNoise}
\end{equation}
where $(X_i)_{i \geq 0}$ are independent and identically random variables, $h$ is a measurable function and $\{N(t)\}_{t>0}$ is a renewal process; see Section 4 for details. In this context, a famous risk indicator is the ruin probability on a finite-time horizon that is the probability that the supremum of the process $S$ exceeds a threshold on a time window $[0,T]$, for a given $T>0$. It is straightforward that maxima necessarily occur on the embedded chain and it is enough to study the discrete-time random sequence $S(N(T)):=(S(T_1),S(T_2),\ldots,S(T_{N(T)}))$, which is of random length $N(T)$. Then, instead of dealing with the extremal behavior of $\left\lbrace S(t)\right\rbrace_{t \leq T}$, we only need to understand the extremal behavior of $\Vert S(N(T)) \Vert_{\infty}$.

We go further and point out that many risk measures in non-life insurance mathematics and in dietary risk assessment can be analyzed from the tail behavior of $\Vert C(N) \Vert$ where $C(N)=(C_1,\ldots,C_N)$ is a random sequence of random length $N$ and $\Vert\cdot\Vert$ is a norm such that  
\begin{equation}
\Vert \cdot \Vert_{\infty}\leq \Vert \cdot \Vert \leq \Vert \cdot \Vert_{1}.  \label{relationnorm} 
\end{equation} 
Thus we consider discrete-time  
processes $C(N)=(C_1,\ldots,C_N)$ where for all $1\leq i\leq N$, $C_{i}\in\mathbb{R}^{+}$ and $N$ is an
integer-valued random variable (\textit{r.v.}), independent of the $C_{i}$'s.
We are interesting in the case when the $C_{i}$'s are regularly varying random
variables. We restrict ourselves to the process $C(N)$ which can be written
in the form
\begin{equation}
C(N)=A(N)X(N), \label{process}%
\end{equation}
where $X(N)=(X_{1},\ldots,X_{N})^{\prime}$ is a random vector with identically distributed marginals which are not necessarily independent and $A(N)$ is
a random matrix of random size $N\times N$ independent of the entries $X_{i}$ of the vector $X(N)$
and of $N$. However, $X(N)$ and $A(N)$ are still
dependent through $N$ that determines their dimensions. Notice that $C(N)$ covers a
wide family of processes with possibly dependent $(X_i)_{i \geq 0}$.  

Our main objectives are: to define regular variation properties for a random length sequence of random variables and to study the spectral properties in order to develop risk
measures. As it will become
clear later, the randomness of the size $N$ of the vector $C(N)$ makes it difficult
to use the common definition of multivariate regular variation in terms of vague convergence; see \cite{Hult}. 
Indeed, to handle regular variation of a random sequence of random length, we need to define a regular variation framework for an infinite-dimensional space of sequences defined on $(\mathbb{R}^+)^{\mathbb{N}^*}$. 
We tackle the problem using the notion of $\mathbb{M}$-convergence, introduced recently in \cite{Lindskog}. A main difference with the finite-dimensional case is that the choice of the norm matters as it determines the infinite-dimensional space to consider; see \cite{Basrak and Segers}, \cite{Basrak:2002},
\cite{Resnick (2002)} and \cite{Resnick (2004)} for a comprehensive review of finite-dimensional multivariate regular variation theory.  

The key point of our approach is the use of a norm satisfying \eqref{relationnorm} that allows to build regular variation \textit{via} polar coordinates. This approach combined with an extension of Breiman Lemma leads to characterize the regular variation properties of $C(N)$; see \cite{Cline} for the statement and the proof of Breiman Lemma. According to the choice of the norm, it characterizes several risk indicators for a large family of processes.

For a particular class of
Shot Noise Processes, we recover the result of \cite{LiWu} regarding the tail behavior when $N$ is a Poisson process and when the $(X_i)_{i\geq 0}$ are asymptically independent and identically distributed. We give also the ruin probability for shot noise processes defined as \eqref{ShotNoise} when the $(X_i)_{i\geq 0}$ are not necessarily independent. Moreover, we shall supplement the information missing by these two indicators by suggesting new ones; see Section 4 for details. We first turn our interest to the
\textit{Expected Severity}, a widely-used risk indicator in
insurance. It is an alternative to Value-at-Risk that is more sensitive to the losses in the tail of the distribution. Then, we shall introduce an indicator called
\textit{Integrated Expected Severity} which supplies information on the total losses themselves. Lastly, our focus will shift to the \textit{Expected Time Over a Threshold},
which corresponds to the average time spent by the process above a given threshold. 

The paper is constructed as follows:  firstly, we specify the framework and
describe the assumptions on the model. In Section 3, we define
regular variation for a random length sequence of random variables, followed by an extension
of Breiman Lemma. Next, we attempt to apply and develop risk indicators for a particular class of
stochastic processes: the Shot Noise Processes. We also apply our final result to a model used in dietary risk assessment called Kinetic Dietary Exposure Model introduced in \cite{Bert_2008}. Finally, Section 5 is devoted to
proving the main results of this paper.

\section{Framework and assumptions} 
\subsection{Regular variation in $c_{\Vert\cdot\Vert}$}
In order to provide an extension of Breiman Lemma for the norm of a sequence $C(N)=(C_1,\ldots,C_N)$ defined as in \eqref{process}, we define regular variation
for any vector $X(N)=(X_1, X_2,\ldots,X_{N},0,\ldots)$ when $N$ is a strictly positive random integer. We denote by $c_{00}$ the set of sequences with
a finite number of non-zero components. To build a regular variation framework on this space seems to be very natural for our purpose but we cannot do it since $c_{00}$ is not complete. Indeed, it is not closed
in $(l_{\infty},\Vert\cdot\Vert_{\infty})$. For example, consider the sequence of $c_{00}$ with elements $u_{i}=(1,1/2,1/3,\ldots,1/i,0,0,\ldots)$, $i\in\mathbb{N}$. Its limit as $i \rightarrow \infty$ is not an element of $c_{00}$. Therefore, we will
choose the completion of $c_{00}$ in $l_{\infty}$ space denoted by {\textbf{$c_{\Vert\cdot\Vert}$}}. In the rest of the paper, we write $\Vert \cdot \Vert$ when the norm satisfies \eqref{relationnorm}. 

\begin{definition} The space $c_{\Vert\cdot\Vert}$ is
the completion of $c_{00}$ in $(\mathbb{R}^{+})^{\mathbb{N}}$ equipped with the convergence in the
sequential norm $\Vert\cdot\Vert$.
\end{definition}

For example, $c_{\Vert\cdot\Vert_{\infty}}=c_{0}$, the space of sequences whose
limit is $0$:
\[
c_{0}:=\left\{  u=(u_{1},u_{2},\ldots)\in(\mathbb{R}^{+})^{\mathbb{N}%
}:\ \underset{i\rightarrow+\infty}{\lim}u_{i}=0\right\}  .
\]
Furthermore, $c_{\Vert\cdot\Vert_{1}}=\ell^{1}(\mathbb{R}^{+})$, the space of
sequences such that $\sum_{i\geq0}u_{i}<\infty$. In the following, we equip $c_{\Vert \cdot \Vert}$ with the
canonical distance $d$ generated by the norm such that for any sequence $X,Y\in
c_{\Vert \cdot \Vert}$, $d(X,Y)=\Vert X-Y \Vert$  and we denote by $\mathbf{0}=(0,0,\ldots)$ the null element in $c_{\Vert \cdot \Vert}$. Besides, for all $j\geq 1$, we denote $e_{j}=(0,\ldots,0,1,0,\ldots)\in c_{\Vert \cdot \Vert}$ the canonical basis of $c_{\Vert \cdot \Vert}$ and we denote by $S(\infty)$, the unit sphere over $c_{\Vert \cdot \Vert}$ defined as $S(\infty)=\lbrace X \in c_{\Vert \cdot \Vert}: \; \Vert X \Vert =1 \rbrace$.  As $c_{\Vert \cdot \Vert}$ is a Banach space, the notion of weak convergence holds on $c_{\Vert \cdot \Vert}$ and one can also define regular variation as in \cite{Hult}.

\begin{definition}\label{def:rv} A sequence $X=(X_1,X_2,\ldots)\in c_{\Vert \cdot \Vert}$ is regularly varying if
there exists a non-degenerate measure $\mu$ such that
\[
\mu_{x}(\cdot)=\frac{\mathbb{P}(x^{-1}X\in\cdot)}{\mathbb{P}(\Vert
X\Vert>x)}\underset{x\rightarrow\infty}{\overset{\nu}{\longrightarrow}}%
\mu(\cdot),
\]
where $\overset{\nu}{\longrightarrow}$ refers to the vague convergence on the
Borel $\sigma-$field of $c_{\Vert \cdot \Vert}\backslash\{\mathbf{0}\}$.
\end{definition}

Here we want to avoid the approach using the vague
convergence because it implies to find compact sets which may be complicated in
several cases, especially in our infinite-dimensional framework.
We use instead the "$\mathbb{M}$-convergence" as introduced by \cite{Lindskog}, which instead of working with compact sets deals with special "regions" that are bounded
away (BA) from the cone $\mathbb{C}$ we choose to remove. These BA sets will replace the "relatively compact sets" generating vague convergence.

We use the same notation as in \cite{Lindskog}, consider the unit sphere $S(\infty)$ over $c_{\Vert \cdot \Vert}$ and we take $\mathbb{S}=c_{\Vert \cdot \Vert}$. We choose to remove $\mathbb{C}= \left \lbrace \textbf{0} \right \rbrace$ which is
a closed space in $c_{\Vert \cdot \Vert}$ and in particular a cone. Then
$\mathbb{S}_{\mathbb{C}}=\mathbb{S}\backslash \mathbb{C}=c_{\Vert \cdot \Vert}\backslash \left\lbrace \textbf{0} \right\rbrace$ is an open set and still a cone. Similarly, we take
$\mathbb{S}^{\prime}=[0,\infty]\times S(\infty)$ and we remove $\mathbb{C}^{\prime}=\{0\}\times
S(\infty)$ which is closed in $[0,\infty]\times S(\infty)$ too. Moreover, for any
$X\in c_{\Vert \cdot \Vert}$, we choose the standard scaling function $g:(\lambda
,X)\longrightarrow\lambda X:=(\lambda X_{1},\lambda X_{2},\ldots)$, $\lambda>0$, which is well suited for polar coordinate
transformations. For $\mathbb{S}^{\prime}$, we choose the scaling function $(\lambda
,(r,W))\longrightarrow(\lambda r,W)$ in order to make $\{0\}\times S(\infty)$ a
cone. Let $h$ be the polar transformation such that $h:c_{\Vert \cdot \Vert}\backslash
\ \{\textbf{0}\}\longrightarrow [0,\infty]\times S(\infty)$ and for all $X\in
c_{\Vert \cdot \Vert}\backslash\{\textbf{0}\}$,
\[
h:X\longrightarrow(||X||,||X||^{-1}X).
\]

As $c_{\Vert \cdot \Vert}$ is a separable Banach space, we can apply
\cite{Lindskog} to define regular variation on this space from the $\mathbb{M}$-convergence. 
Precisely, condition (4.2) is satisfied and Corollary 4.4 in \cite{Lindskog} holds. Combining with Theorem 3.1 in \cite{Lindskog} which ensures the homogeneity property of the limit measure $\mu$ such that $\mu(\lambda A)=\lambda^{-\alpha}\mu(A)$, for some $\alpha >0$, $\lambda>0$ and $A \in c_{\Vert \cdot \Vert} \backslash \{ \textbf{0} \}$, it leads to the following characterization of regular variation in $c_{\Vert \cdot \Vert}$. 
\begin{proposition}\label{RegVarX}
A sequence of random elements $X=(X_{1},X_{2},\ldots)\in c_{\Vert \cdot \Vert}  \backslash \{ \mathbf{0 }\}$ is regularly
varying iff $\Vert X\Vert$ is regularly varying and
\begin{equation*}
\mathcal{L}\left(  \Vert X \Vert^{-1}X\mid\Vert X\Vert>x\right)
\underset{x\rightarrow\infty}{\longrightarrow}\mathcal{L}(\Theta),
\end{equation*}
for some random element $\Theta\in S(\infty)$ and $\mathcal{L}(\Theta)=\mathbb{P}(\Theta \in \cdot)$ is the spectral
measure of $X$.  
\end{proposition}
It means that the regular variation of $X$ is completely characterized by the tail index $\alpha$ of $\Vert X \Vert$ and the spectral measure of $X$.

\subsection{Assumptions}
In order to get the spectral measure of any random sequence $C(N)$ of random length $N$ defined as in \eqref{process}, we require the following conditions:

\begin{itemize}
\item[\textbf{(H0)}] \textbf{Length:} $N$ is a positive integer-valued
\textit{r.v.} such that $\mathbb{E}[N]>0$ and admits moments of order $ 2+ \alpha+\epsilon  $, $\epsilon>0$.

\item[\textbf{(H0')}] \textbf{Poisson} \textbf{counting process:} $N$ is an
inhomogeneous Poisson Process with intensity function $\lambda(\cdot)$ and
cumulative intensity function $m(\cdot)$.

\item[\textbf{(H1)}] \textbf{Regular variation:} The $(X_{i})_{i\geq
0}$ are identically distributed (\textit{i.d.}) with common mean
$\gamma$ and cumulative distribution function (\textit{c.d.f.}) $F_{X}$, such
that the survival function $\overline{F}_{X}=1-F_{X}$ is regularly varying
with index $\alpha>0$, denoted by $X$ $\in RV_{-\alpha}$.

\item[\textbf{(H2)}] \textbf{Uniform} \textbf{asymptotic independence:} We
assume a uniform bivariate upper tail independence condition: for all
$i, j\geq1$,
\[
\underset{i\neq j}{\sup}\left\vert \frac{\mathbb{P}(X_{i}>x,X_{j}%
>x)}{\mathbb{P}(X_{1}>x)}\right\vert \underset{x\rightarrow\infty
}{\longrightarrow}0.
\]

\item[\textbf{(H3)}] \textbf{Regularity of the norm:} The norm $\Vert
\cdot\Vert$ satisfies $\Vert\cdot\Vert_{\infty}\leq\Vert\cdot\Vert\leq
\Vert\cdot\Vert_{1}$. 

\item[\textbf{(H4)}] \textbf{Tail condition on the matrix $A(N)$:} The random
entries $(a_{i,j})$ of $A(N)$ are independent of the $(X_{i})$. Moreover, there exists
some $\epsilon>0$ such that
\[
\mathbb{E}[||A(N)||^{\alpha+\epsilon}N^{1+\alpha+\epsilon}]< \infty,
\]
where $\|\cdot\|$ also denotes the corresponding induced norm on the space of $N$-by-$N$ matrices.

\item[\textbf{(H5)}] \textbf{The matrix $A(N)$ is not null}.
\end{itemize}

Let us discuss the assumptions. The condition \textbf{(H1)} implies that the regular variation of the sequence $C(N)=(C_1,\ldots,C_N)$ defined as in \eqref{process} comes from the regular variation of the sequence $X(N)$ and \textbf{(H2)} means in addition that the probability of two components of the sequence $X(N)$ exceeding a high threshold goes to $0$. Combining \textbf{(H1)} and \textbf{(H2)}, it appears that the regular variation of the sequence $C(N)$ is mostly due to the regular variation of one of the component of $X(N)$. Note that if the $X_{i}$'s are exchangeable and asymptotically independent, then
\textbf{(H2)} holds. An example of a time series satisfying \textbf{(H2)} is a stochastic volatility model defined by
$$X_t= \sigma_t \epsilon_t, \quad t \in \mathbb{Z}, $$ 
where the innovations $\epsilon_t$ are standardized positive \textit{i.i.d. r.v.}'s such that $\mathbb{P}(\epsilon_0\neq 0)>0$ and $\mathbb{E}( \epsilon_t^{1+\delta})$ holds for some $\delta>0$. The volatility $\sigma_t$ satisfies the equation 
$$\log(\sigma_t)=\phi \log(\sigma_{t-1})+ \xi_t, \quad t \in \mathbb{Z},$$
with $\phi \in (0,1)$, the innovations $(\xi_t)_{t \in \mathbb{Z}}$ are \textit{i.i.d.} \textit{r.v.}'s independent of $(\epsilon_t)_{t \in \mathbb{Z}}$ such that $\mathbb{E}(\xi_0^2)<\infty$ and $\mathbb{P}(\xi_t>x) \sim K x^{-\alpha}e^{-x} $ when $x \rightarrow \infty$ for some positive constants $K$ and $\alpha \neq 1$. This is a particular case of stochastic volatility models with Gamma-Type Log-Volatility when the log-volatility is a AR(1) model. One can check that $\sigma_t$  and $X_t$ are regularly varying with index $\alpha=1$. Moreover, $(\log(\sigma_0), \log(\sigma_h))$ is asymptotically independent as well as $(\sigma_0, \sigma_h)$ and $(X_0,X_h)$ for any strictly positive $h$; see \cite{anja} for details.     

Besides, \textbf{(H3)} implies $\ell^{1}%
(\mathbb{R}^{+})\subseteq c_{\Vert\cdot\Vert}\subseteq c_{0}$ and holds for any
$\ell^{p}$ norm for $1\leq p\leq\infty$. This condition is always
assumed in the paper, even if many conditions could be weakened if we only considered the norm $\Vert \cdot \Vert_{\infty}$. But our aim is to develop a method which can be applied for a broad class of processes and norms.  Finally, \textbf{(H4)} requires more finite moments on $N$ than  \textbf{(H0)}. This condition can be viewed as an extension of
the moment condition of the Breiman's lemma. It is required to estimate the
tail of $\Vert C(N)\Vert=\Vert A(N)X(N)\Vert$ from the one of $\Vert
X(N)\Vert$ and it is not  restrictive in practice when $N$ satisfies \textbf{(H0')}.
\begin{lemma}\label{lem:h4}
Assume that {\bf (H3)} holds and that $A(N)=(a_{i,j})_{1 \leq i,j \leq N}$ has \textit{ i.d} components satisfying $\mathbb{E}[|a_{1,1}|^p]< \infty$ for some  $p>\alpha\vee 1$. If $\E[N^{2p+1}]<\infty$ then  \textbf{(H4)} holds.
\end{lemma} 
\begin{proof}From \textbf{(H3)} and convexity, we have 
\begin{align*}
\mathbb{E}\left[ \Vert A(N) \Vert^p N^{p+1} \right] & \leq \sum\limits_{n=1}^{\infty} \mathbb{E}\left[ \Vert A(N) \Vert_1^p N^{p+1}  \mathrm{1}_{\{N=n\}}\right]\\
&  \leq \sum\limits_{n=1}^{\infty}\mathbb{E}\left[ \left( \sup_{1\le j\le n}\sum\limits_{i=1}^{n} |a_{i,j}| \right)^p \right]n^{p+1}\mathbb{P}(N=n)\\
&  \leq \sum\limits_{n=1}^{\infty}n\mathbb{E}\left[ \left(  \sum\limits_{i=1}^{n} |a_{i,1}| \right)^p \right]n^{p+1}\mathbb{P}(N=n)\\
&  \leq \sum\limits_{n=1}^{\infty}n^{p-1}\mathbb{E}\left[  |a_{1,1}|^p   \right]n^{p+2}\mathbb{P}(N=n)\\
&  \leq \mathbb{E}[|a_{i,j}|^p] \sum\limits_{n=1}^{\infty}n^{2p+1} \mathbb{P}(N=n)< \infty
\end{align*}
and then \textbf{(H4)} holds.\end{proof}

\section{Regular variation of random sequences of random length}
\subsection{Regular variation properties in
$c_{\Vert \cdot \Vert}$ under \textbf{(H0)}}
We focus in providing regular variation properties for a sequence $X(N)=(X_{1},X_{2},\ldots,X_{N},0,\ldots)$ under \textbf{(H0)}, written $X(N)=(X_{1},X_{2},\ldots,X_{N})$ for short in the sequel. We write ${\bf 0}$ the null element in $c_{\Vert \cdot \Vert}$. The sequence $X(N)$ can be seen as an element of $c_{\Vert \cdot \Vert}$ whose the number of non-null elements is driven by the random variable $N$.

\begin{proposition}\label{MultRegVar} 
A sequence of random elements $X(N)=(X_{1},X_{2},\ldots,X_{N})\in
c_{\Vert \cdot \Vert} \backslash  \lbrace {\bf 0} \rbrace$ for $N$ satisfying \textbf{(H0)} is regularly varying if the random variable $\Vert X(N)\Vert$ is regularly varying and
\begin{equation*}
\mathcal{L}\left(  \Vert X(N)\Vert^{-1}X(N)\mid\Vert X(N)\Vert>x\right)
\underset{x\rightarrow\infty}{\longrightarrow}\mathcal{L}(\Theta(N)),
\end{equation*}
for some random element $\Theta(N)\in S(\infty)$. The distribution of $\Theta(N)$ is the spectral measure of $X(N)$.
\end{proposition}
\begin{proof}
The random variable $N$ is necessarily not null thanks to \textbf{(H0)}. Then the proof is straightforward from Proposition \ref{RegVarX}.
\end{proof}

Notice that $\Theta(N)\in S(\infty)$  is an infinite sequence $\Theta(N)=(\Theta_1(N), \Theta_2(N),\ldots)$.
The following proposition is relevant for the results of this paper. It is a first example of such regularly random vectors of random length under \textbf{(H3)}. Besides, it is an extension of Lemma A6 in \cite{Weng}.

\begin{proposition}\label{regvect}
Let $X(N)=(X_{1},\ldots,X_{N})\in c_{\Vert \cdot \Vert} $ such that \textbf{(H0)-(H3)} hold. Then
we have
\begin{equation*}
\lim_{x\rightarrow\infty}\frac{\mathbb{P}(\Vert X(N)\Vert>x)}{\mathbb{P}%
(X_{1}>x)}=\mathbb{E}[N]>0. 
\end{equation*}
\end{proposition} 
\begin{proof}
See Section 5.
\end{proof}
Here the moment condition in \textbf{(H0)} is required to handle the $\Vert \cdot \Vert_1$ case. Releasing the assumption \textbf{(H0)}, it is easy to draw an \textit{i.i.d.} sequence of random length which is regularly varying in $c_{\Vert \cdot \Vert_{\infty}}$ but not in $c_{\Vert \cdot \Vert_1}$. For instance, for the $\Vert \cdot \Vert_{\infty}$ case, the result is true under the integrability of $N$ only but this assumption is not sufficient to ensure the convergence for the $\Vert \cdot \Vert_1$ case. Note also that \textbf{(H3)} is required  since the convergence might not hold for some norms that do not satisfy \textbf{(H3)}. 
From this proposition, a characterization of the spectral measure of $X(N)$
is possible.

\begin{proposition}\label{Caracspecmea}
\label{prop:id} If \textbf{(H0)}-\textbf{(H3)} hold then $X(N)=(X_{1}%
,\ldots,X_{N}) \in c_{\Vert \cdot \Vert} \backslash \lbrace {\bf 0} \rbrace$ is regularly varying in the sense of Proposition \ref{MultRegVar} and its spectral measure is
characterized by
\begin{align*}
\mathbb{P}%
(\Theta(N)=e_{j})=\frac{\mathbb{P}(N\geq j)}{\mathbb{E}[N]},\qquad j\ge 1.
\end{align*} 
\end{proposition}

\begin{proof}
See Section 5.
\end{proof}

The spectral measure $\Theta(N)$ of $X(N)$ belongs to $S(\infty)$, the unit sphere on the infinite-dimensional space $c_{\Vert \cdot \Vert}$. We cannot draw straightforwardly the parallel between regular variation in $\mathbb{R}^n$ and in $c_{\Vert \cdot \Vert}$ and $\Theta(N)$ needs to be handle carefully. We have not achieved to find a way to define properly the distribution of $\Theta(N)$ given that $N=n$. Following arguments in \cite{Lindskog} to handle regular variation of the infinite-dimensional space $\mathbb{R}^{\infty}_+$ from regular variation of the finite-dimensional space $\mathbb{R}_{+}^p$, $p>0$, the natural choice would have been to consider the projection operator $Proj_n:c_{\Vert \cdot \Vert} \rightarrow \mathbb{R}^n$ defined as $Proj_n(X=(X_1,X_2,\ldots))=(X_1,X_2,\ldots,X_n)$ and to define the distribution of $\Theta(N)$ given $N=n$ as $\mathcal{L}(\Theta(N)|N=n)=\mathcal{L}(Proj_n(\Theta(N))$. Then, for all $j \leq n$, we would have had 
$$\mathbb{P}(\Theta(N)=e_j | N=n)=\dfrac{\mathbb{P}(\Theta(N)=e_j)}{\sum\limits_{j=1}^n \mathbb{P}(\Theta(N)=e_j)}=\dfrac{\mathbb{P}(N \geq j)}{\sum\limits_{j=1}^n \mathbb{P}(N \geq j)},$$
which is not coherent with the distribution of $\Theta(N)$ provided in Proposition \ref{Caracspecmea} and this representation does not work. This provides further evidence that $\Theta(N)$ cannot admit the representation $(\Theta_1(N),\ldots,\Theta_N(N))$ and must be considered as an infinite-dimensional random element.  Nonetheless, note that $\sum_{j=1}^{\infty} \mathbb{P}(\Theta(N)=e_j)=1$ and then $\Theta(N)= \bigcup_{j=1}^{\infty} e_j $ almost surely, thus we fully characterized the spectral measure of $X(N)$ in Proposition \ref{Caracspecmea}.  

\begin{remark}\label{rem:Thetaiid}
When $N=n$ for a fixed $n \geq 1$, it follows that for all $1 \leq j \leq n$, $\mathbb{P}(\Theta(n)=e_j)=n^{-1}$. \end{remark}

\subsection{Generalization to the matrix product}

We generalize
this approach to sequences in $c_{\Vert \cdot \Vert}$ defined from the product of $X(N)$ by a random matrix $A(N)=(a_{i,j})_{1\leq i,j\leq
N}$ of random size $N\times N$. We denote this vector $C(N)=A(N)X(N)$,
which is of length $N$. We keep the previous
notation and we denote by $A_k(N)$ the $k$-th column of the matrix $A(N)$. Here and in what follows, we work under the assumptions \textbf{(H3)} and \textbf{(H4)}.
Notice that \textbf{(H3)} implies that the canonical basis of $c_{\Vert \cdot \Vert}$ denoted by
$(e_{i})_{i\ge 1}$ is standardized, \textit{i.e.} for all $i\geq1$,
$\Vert e_{i}\Vert=1$.

Note that on $N=n$, $A(n)$ and $X(n)$ are independent. Then we directly deduce from Remark \ref{rem:Thetaiid} and the multivariate Breiman's lemma (see Proposition 5.1 in \cite{Basrak:2002a}) the following useful proposition.
\begin{proposition}\label{cor:convprocess}
Let $A(N)$ and $X(N)$ be defined as above and assume \textbf{(H0)-(H3)}.
Then, for any $n\ge 1$, we have
\begin{equation*}
\mathbb{P}(\Vert x^{-1} C(n) \Vert>x)/\mathbb{P}(X_1>x) \underset
{x\rightarrow\infty}{\longrightarrow}\E\left[\sum_{k=1}^n\|A_k(n)\|^\alpha\right].
\end{equation*}
Moreover $C(n)$ is regularly varying under {\bf (H5)}. 
\end{proposition}

\begin{proof}
From Proposition 5.1 in \cite{Basrak:2002a}, we have  
$$\mathbb{P}(x^{-1}A(n)X(n) \in \cdot)/ \mathbb{P}(\Vert X(n) \Vert>x) \underset
{x\rightarrow\infty}{\longrightarrow} \mathbb{E}[\mu \circ A(n)^{-1}(\cdot)]$$
with $\mu$ a radon measure defined by 
$$\mathbb{P}(x^{-1}X(n)\in \cdot)/\mathbb{P}(\Vert X(n) \Vert>x) \underset
{x\rightarrow\infty}{\longrightarrow}\mu(\cdot).$$

Then, 
$$\mathbb{P}(\Vert x^{-1}A(n)X(n)\Vert \in \cdot)/ \mathbb{P}(\Vert X(n) \Vert>x) \underset
{x\rightarrow\infty}{\longrightarrow} \mathbb{E}[\mu \circ \Vert A(n)\Vert ^{-1}(\cdot)].$$

By homogeneity of $\mu$, it follows that 
$$\mathbb{P}(\Vert A(n)X(n)\Vert)>x )/ \mathbb{P}(\Vert X(n) \Vert>x) \underset
{x\rightarrow\infty}{\longrightarrow} \mathbb{E}[\mu \{x:\Vert A(n)x \Vert>1]=\mathbb{E}[\Vert A(n)\Theta(n) \Vert^{\alpha}].$$
 
From Remark \ref{rem:Thetaiid}, $\Theta(n)= \bigcup_{i \leq n}e_i$ and $\mathbb{P}(\Theta(n)=e_i)=n^{-1}$ for all $i\leq n$ which leads to 
$$\mathbb{E}\left[ \Vert A(n) \Theta (n) \Vert ^{\alpha}\right]=
n^{-1}\E\left[\sum_{k=1}^n\|A_k(n)\|^\alpha\right],$$

Then, from Lemma \ref{Conicreg}, we have

$$\mathbb{P}(\Vert A(n)X(n)\Vert)>x )/ \mathbb{P}( X_1>x) \underset
{x\rightarrow\infty}{\longrightarrow} \E[\sum_{k=1}^n\|A_k(n)\|^\alpha],$$

which concludes the proof. \end{proof}

This proposition plays a leading role in the proof of the following theorems, which generalize the Breiman's lemma to random length sequences of \textit{r.v.'s} asymptotically independent and identically distributed.

\begin{theorem}\label{Breiman}
Let \textbf{(H0)-(H4)} hold and $C(N)=A(N)X(N)$, then we have
\begin{equation*}
\lim_{x\rightarrow\infty}\frac{\mathbb{P}(\Vert C(N)\Vert>x)}{\mathbb{P}
(X_{1}>x)}=\mathbb{E}\left[
{\sum\limits_{k=1}^{N}} ||A_{k}(N)||^{\alpha}\right]  .
\end{equation*}
\end{theorem}

The proof is postponed to Section 5. Notice that Theorem \ref{Breiman} holds if $A(N)$ does not necessarily
satisfy \textbf{(H5)} and then we allow that $\mathbb{P}(\Vert
C(N)\Vert>x)/\mathbb{P}(X_1 \geq x)\rightarrow0$ when $x$ goes to infinity; see Section 5 for details. Under the additional assumption \textbf{(H5)}, we are now ready to prove that $C(N)$ is regularly varying.

\begin{theorem}\label{th:regvar}
If \textbf{(H0)-(H5)} hold, $C(N)=A(N)X(N)$ is regularly
varying and its spectral measure is given by
\begin{align*}
\mathbb{P}\left(  \Vert C(N)\Vert^{-1}C(N)\in\cdot\mid\Vert C(N)\Vert
>x\right)  \underset{x\rightarrow\infty}{\longrightarrow}   &  \frac{\mathbb{E}\left[  \sum_{k=1}^{N}\Vert A_{k}(N)\Vert^{\alpha
}\mathrm{1}\hspace{-0.35em}\mathrm{1}_{\Vert A_{k}(N)\Vert^{-1}A_{k}%
(N)\in\cdot}\right]  }{\mathbb{E}\left[  \sum_{k=1}^{N}\Vert A_{k}%
(N)\Vert^{\alpha}\right]  }
\end{align*}
\label{ChaSpectrMeasure}

\end{theorem}

The proof is postponed to Section 5. Although the characterization is common for any norm such that \textbf{(H3)} holds, the result essentially depends on the choice of the norm. Despite this remark, it is noteworthy that the spectral measure can
be described in a unified way even if it belongs to different spaces, regarding the choice of the norm.

In the sequel, we assume that \textbf{(H0)-(H5)} hold. Following \cite{Basrak and Segers}, for all $i \leq N$, we have 
\begin{equation*}
\lim_{\epsilon \rightarrow 0} \underset{x\rightarrow\infty}{\lim} 
\mathbb{P}(\vert C_i(N) \vert >x\epsilon \; \vert \; \vert\vert C(N) \vert\vert >x) = \mathbb{E}\left[  \vert \Theta_i(N)\vert^{\alpha} \right].
\end{equation*}
From Theorem \ref{th:regvar}, we have 
\begin{equation*}
\mathbb{E}\left[ \vert\Theta_i  \vert^{\alpha} \right]=
\dfrac{\mathbb{E}\left[ \sum_k \vert\vert A_k(N) \vert\vert^{\alpha} \times   \dfrac{\vert a_{i,k}\vert^{\alpha}}{\vert\vert A_k(N) \vert\vert^{\alpha}}\1_{\{i\le N\}}\right]}{\mathbb{E}\left[ \sum_k \vert\vert A_k(N) \vert\vert^{\alpha}\right]}\\
=\dfrac{\mathbb{E}\left[ \sum_k \vert a_{k,i}\vert^{\alpha}\1_{\{i\le N\}}\right]}{\mathbb{E}\left[ \sum_k \vert\vert A_k(N) \vert\vert^{\alpha}\right]}.
\end{equation*}
Let us consider the asymptotically independent case when $a_{k,i}=\1_{\{i=k\}}$ then, 
\begin{align*}
\mathbb{E}\left[ \vert \Theta_i(N) \vert^{\alpha} \right]
=\dfrac{\mathbb{E}\left[ \sum\limits_{k=1}^N \1_{\left\lbrace i=k \right\rbrace} \right]}{\mathbb{E}[N]}=\dfrac{\mathbb{P}(N \geq i)}{\mathbb{E}(N)}.
\end{align*}
It is a mean constraint on the spectral measure. To be precise, if the margins of the random vector are identically distributed, then necessarily $\mathbb{E}\left[ \vert \Theta_i(N) \vert^{\alpha} \right]=\mathbb{P}(N \geq i)/\mathbb{E}(N)$, $i\ge 1$. We recover the mean constraint $\mathbb{E}\left[ \vert \Theta_i(n) \vert^{\alpha} \right]=n^{-1}$, $1\le i \le n$ when $n$ is fixed; see \cite{Gudendorf}.

\section{Applications}

This section is devoted to calculating the constant
$\sum_{k=1}^N ||A_{k}(N)||^{\alpha}$ for various examples. To obtain explicit results, we assume in this part that $N$ is defined as in \textbf{(H0')} which implies in particular that the moment condition \textbf{(H4)} holds.
The result is derived thanks to the order statistics property of a Poisson process (see \cite{Mikosch}, Section 2.1.6). The computation of the
constant $\sum_{k=1}^N ||A_{k}(N)||^{\alpha}$ for different norms and
different matrices $A(N)$ permits to develop various risk measures.

As mentioned before, $C(N)$ covers a wide family of processes. We deal here with
an example of a \textit{Shot Noise Process} (\textit{SNP}). \textit{SNP}
were first introduced in the \textit{20's} to study the fluctuations
of current in vacuum tubes and used in an insurance context from the
second half of the twentieth century; see \cite{Weng} and
\cite{Schmidt} for more details on the \textit{SNP} theory. We
restrict ourselves to the study of particular \textit{SNP'}s
defined by
\begin{equation}
Y(t)=\sum\limits_{i=1}^{N(t)}X_{i}\times h_{i}(t,T_{i}),\qquad\forall\ t\geq0,
 \label{SNP}%
\end{equation}
where $(h_{i})_{i\geq0}$ are \textit{i.i.d.} non-negative measurable random functions called "\textit{shock functions}", which are 
independent of the shocks $(X_{i})_{i\geq 0}$.  The "shock instants" $(T_i)_{i \geq 0}$ are \textit{r.v.}'s independent of $(X_i)_{i \geq 0}$ such that for all $i \geq 0$, $T_i=\sum_{0 \leq k \leq i} \Delta T_k$, where $(\Delta T_k)_{k \geq 0}$ is an \textit{i.i.d.} sequence of \textit{r.v.}'s called "inter-arrivals". Under \textbf{(H0')}, the inter-arrivals are exponentially distributed.
In this context, $N(t)=\left\lbrace \# i:T_i \leq t \right\rbrace$ is a renewal process which counts the number of claims that occurred until time $t$; see \cite{Mikosch} for a survey on renewal theory. We assume that for all $i \geq 0$, $t>0$ and $\alpha > 0$, there exists $\epsilon>0$ such that $\mathbb{E}[h_i^{\alpha + \epsilon}(t,T_i)]< \infty$.

In the sequel, we
denote by $\overline{F_{X}^{I}}$ the integrated tail distribution associated
to the \textit{r.v.} $X$, which is defined for all $y>0$ by
\[
\overline{F_{X}^{I}}(y)=\frac{1}{\mathbb{\gamma}} \int
\limits_{y}^{+\infty} \overline{F}_{X}(x)dx,
\]
with $\gamma=\mathbb{E}(X)$. We write $N(t)$ instead of $N$ to stress the fact we are dealing with counting processes and therefore we study the process through time. 

\subsection{Asymptotic tail behavior of \textit{SNP}'s}

We first apply our method to determine the asymptotic behavior of a
\textit{SNP} defined as \eqref{SNP} as a corollary of our main result.

\begin{corollary}\label{cor:var}
Under {\bf (H0')}-{\bf (H1)}-{\bf (H2)}, assume that the random functions $h_{j}(T,\cdot)$'s are \textit{i.i.d.}, independent of
the $T_{j}$'s and integrable of order $p>\alpha$, then%
\begin{equation*}
\lim_{x\rightarrow\infty}\frac{\mathbb{P}(Y(T)>x)}{\mathbb{P}(X_{1}%
>x)}=m(T)\, \mathbb{E}%
[h_{0}^{\alpha}(T,V_{0})], 
\end{equation*}
where $V_{0}$ admits the density $\lambda(t)/m(T)$, $0\leq t\leq T$.
\end{corollary}
This corollary plays a leading role to determine the risk indicators in Section \ref{Riskind}.
Besides, we recover the recent results of \cite{Weng} and \cite{LiWu} on the
tail of $\{Y(T)\}_{T\geq0}$.

\begin{proof}
Let $T$ be a strictly positive constant and let $C(N(T))=A(N(T))X(N(T))$ be a sequence of random length $N(T)$ with $A(N(T))$, the diagonal matrix such that for all $1\leq
j\leq N(T),~a_{jj}=h_{j}(T_{,}T_{j})$. Then, one can write $Y(T)=\sum
_{j=1}^{N(T)}X_{j}h_{j}(T,T_{j})=\Vert C(N(T))\Vert_{1}$. Using similar arguments than in the proof of Lemma \ref{lem:h4}, we check that {\bf (H4)} holds and we apply
Theorem \ref{Breiman} in order to obtain
\[
\lim_{x\rightarrow\infty}\frac{\mathbb{P}(Y(T)>x)}{\mathbb{P}(X_{1}%
>x)}=\mathbb{E}\left[ \sum\limits_{k=1}^{N(T)}\Vert A_{k}%
(N(T))\Vert_{1}^{\alpha}\right]  .
\]
Note that for all $k\leq N(T)$, $||A_{k}(N(T))||_{1}^{\alpha}=h_{k}^{\alpha
}(T_{,}T_{k})$. Let $V_{(k)}~$be the $k$-th order statistic associated to
the \textit{i.i.d.} sequence $(V_{0},V_{1},\ldots)$ distributed as $V_{0}$. From
\cite{Mikosch}, Section 2.1.6, $(T_{k}\ |\ N(T))\overset{d}{=}V_{(k)}$ and it follows that
\begin{align*}
\mathbb{E}\left[  {\sum\limits_{k=1}^{N(T)}}h_{k}^{\alpha}(T_{,}T_{k})\right]
&  =\mathbb{E}\left[  \mathbb{E}\left[  {\sum\limits_{k=1}^{N(T)}}h_{k}^{\alpha
}(T_{,}T_{k})\ |\ N(T)\right]  \right]\\
&  =\mathbb{E}\left[  \mathbb{E}\left[
{\sum\limits_{k=1}^{N(T)}}h_{k}^{\alpha}(T_{,}V_{(k)})\ |\ N(T)\right]  \right]  \\
&  =\mathbb{E}\left[  \mathbb{E}\left[  {\sum\limits_{k=1}^{N(T)}}h_{k}^{\alpha
}(T_{,}V_{k})\ |\ N(T)\right]  \right]\\  
&  =\mathbb{E}[N(T)]\mathbb{E}[h_{0}^{\alpha}(T,V_{0})],
\end{align*}
which is the desired result.
\end{proof}
Notice that the asymptotic behavior of $Y(T)$ as $T\to\infty$ relies on the one of the shock function $h_0$. One case corresponds to $\E[N(T)]\mathbb{E}
[h_{0}^{\alpha}(T,V_{0})]\to C$ for some constant, then $Y(\infty)$ may be well defined. Thus, the \textit{SNP} may admit a stationary distribution $Y(\infty)=\sum_{i\ge 1}X_{i}\times h_{i}(T,T_{i})$ that is regularly varying similarly than $X_1$. Another case corresponds to $\E[N(T)]\mathbb{E}
[h_{0}^{\alpha}(T,V_{0})]\to \infty$ and then it is very likely that $Y(\infty)=\infty$ a.s.. In the latter case, we have explosive shot noise processes.

\subsection{Ruin probability}

We are
interesting in determining the finite-time ruin probability $\psi$ of a \textit{SNP
}defined as in \eqref{SNP},\textit{ }which is the probability that $Y(t)$
exceeds some given threshold $x\in\mathbb{R}^{+}$ on a period $[0,T]$,
\textit{i.e.}
\[
\psi(x,T)=\mathbb{P}\Big(  \sup_{0\leq t\leq T}Y(t)>x\Big)  ,\quad T>0.
\]

\begin{corollary}\label{cor:ruin}
Assume that the conditions of Corollary \ref{cor:var} hold. If $h_{j}(\cdot,T)$
is a non-increasing function for any $T>0$, then,%
\[
\lim_{x\rightarrow\infty}\frac{\psi(x,T)}{\mathbb{P}(X_{1}
>x)}=m(T)\mathbb{E}%
[h_{0}^{\alpha}(V_{0},V_{0})].
\]
\end{corollary}
Notice that if $h_{j}(\cdot,T)$
is a non-decreasing function for any $T>0$, then the maximum of the \textit{SNP} is achieved at time $T$ and so the ruin probability can be computed thanks to Corollary \ref{cor:var}. An intermediate example is when the random shot function $h_0$ can be either increasing and decreasing; see Section \ref{KDEM}. 
\begin{proof}[Proof of Corollary \ref{cor:ruin}]
Let $T$ be a strictly positive constant. Note that in this setup, the maximum of the process $\{Y(T)\}_{T\geq
0}$ is necessarily reached on its embedded chain $\{Y(T_{i})\}_{i\geq0}$. Then, it is equivalent to the study the tail of $\max_{1\leq k\leq N(T)}
\sum_{j=1}^{k}X_{j}h_{j}(T_{k},T_{j})$. To do so, let $C(N(T))=A(N(T))X(N(T))$ with
$A(N(T))$, the lower triangular matrix such that $a_{k,j}=h_{j}(T_{k},T_{j})$ for
$1\leq j\leq k$ and $a_{k,j}=0$ for $k+1\leq j\leq N$. Now, observe that
$\max_{1\leq k\leq N(T)}\sum_{j=1}^{k}X_{j}h_{j}(T_{k},T_{j})=$ $\Vert
C(N(T))\Vert_{\infty}$. We check that {\bf (H4)} holds thanks to Lemma \ref{lem:h4} and from Theorem \ref{Breiman}, we have%
\[
\lim_{x\rightarrow\infty}\frac{\Vert C(N(T))\Vert_{\infty}}{\mathbb{P}(X_{1}%
>x)}=\mathbb{E}\left[ \sum\limits_{k=1}^{N(T)}||A_{k}%
(N(T))||_{\infty}^{\alpha}\right]  .
\]
Conditionally on $N(T)$, we have $A_{j}(N(T))=(0,\ldots,0,h_{j}(V_{(j)},V_{(j)})),\ldots,h_{j}(V_{(N(T))},V_{(j)}%
))^{\prime}$. Then, $\Vert A_{j}(N(T))\Vert_{\infty}^{\alpha}=h_{j}^{\alpha
}(V_{(j)},V_{(j)})$ and with similar arguments than in the previous corollary,
it follows that
\[
\mathbb{E}\left[  {\sum\limits_{k=1}^{N(T)}}||A_{k}(N)||_{\infty}^{\alpha
}\right]  =\mathbb{E}\left[  N(T)\times\mathbb{E}\left[  h_{0}^{\alpha}%
(V_{0,}V_{0})\right]  \ \right],
\]
which concludes the proof.
\end{proof}
\begin{remark}\label{rem:ruin}
If for $T>0$, $h_0(T,T)=c$ with $c>0$, then Corollary \ref{cor:ruin} extends to any counting process $N$ providing
$\psi(x,T)\sim c^\alpha\, \E[N(T)]{\mathbb{P}(X_{1}>x)}$ as $x\to \infty$. Besides, in such case where $h_0(T,T)$ is constant,  a sandwich argument yields also the desired result. Indeed, we have
\begin{align*}
\mathbb{P}\left( c\max_{i=1,\ldots,N(T)}X_i>x \right) \leq \mathbb{P}\left( \max\limits_{i=1,\ldots,N(T)} \left\lbrace \sum\limits_{k=1}^{i}X_ih_i(T,T_i)\right\rbrace  >x \right) \leq \mathbb{P}\left(c\sum\limits_{i=1}^{N(T)}X_i>x \right).
\end{align*}
So Proposition \ref{regvect} can be directly used to find again that 
\begin{align*}
\psi(x,T) \underset{x \rightarrow \infty}{\sim} c^\alpha\mathbb{E}[N(T)]\mathbb{P}(X_1>x).
\end{align*}
\end{remark}
\begin{remark}
We obtain an asymptotic relation between the tail behavior
and the ruin probability of a process defined as in \eqref{SNP}. Precisely, we have
\[
\mathbb{P}\left(  Y(T)>x\right)  \underset{x\rightarrow\infty}{\sim}
\frac{\mathbb{E}\left[  h_{0}^{\alpha}(T,V_{0})\right]  }{\mathbb{E}
[h_{0}^{\alpha}(V_{0},V_{0})]} \psi(x,T)
\]
when $h_{j}(\cdot,T)$
is a non-increasing function for any $T$ and 
\[
\mathbb{P}\left(  Y(T)>x\right)  \underset{x\rightarrow\infty}{\sim} \psi(x,T)
\]
when $h_{j}(\cdot,T)$
is an non-decreasing function.
\end{remark}

\subsection{Other risk indicators}\label{Riskind}

We propose in this part three indicators to supplement the information given by the ruin
probability and the tail behavior. The ruin probability  permits to know if the process has exceeded the threshold but provides no information about the exceedences themselves or about the duration of the exceedences.
 
To fill the gap, we first bear our interest on the
\textit{Expected Severity}. Then, we present an indicator called
\textit{Integrated Expected Severity}, which provides information on the total of the exceedences. Finally, we are
interested in the \textit{Expected Time Over a Threshold}
which corresponds to the average time spent by the process over a threshold. We keep the previous notation and features on the process defined as \eqref{SNP}.

\subsubsection{The Expected Severity and the Integrated Expected Severity}

Let us first begin with the Expected Severity. By definition, the Expected
Severity for a given threshold $x$, written $ES(x)$, is the quantity dealing
with the mean of the excesses knowing that the process has already reached the
reference threshold $x$, defined for all $T>0$ by $\mathbb{E}[[Y(T)-x]_{+}],$ where $[\cdot]_{+}$ is the positive part function.

\begin{proposition}
Assume that the conditions of Corollary \ref{cor:var} hold. For
any $T>0$, the expected severity of a process defined as \eqref{SNP} is given by
\[
ES(x)\underset{x\rightarrow\infty}{\sim}\gamma m(T)\mathbb{E}[h_{0}^{\alpha}(T,V_{0})]\,\overline{F_{X}
^{I}}(x).
\]

\end{proposition}

\begin{proof}
By definition,
\[
 \mathbb{E}[Y(T)-x]_{+}  =\mathbb{E}\left[
{\sum\limits_{i=1}^{N(T)}}
X_{i}h_{i}(T,T_{i})-x\right]  _{+} ={\int\limits_{x}^{+\infty}}
\mathbb{P}\left(
{\sum\limits_{i=1}^{N(T)}}
X_{i}h_{i}(T,T_{i})>y\right)dy.
\]
From Corollary \ref{cor:var}, it follows that
\begin{align*}
ES(x)\underset{x\rightarrow\infty}{\sim}&
\int\limits_{x}^{+\infty}
\mathbb{E}[N(T)]\mathbb{E}[h_{0}^{\alpha}(T,V_{0})]P(X>y)dy\\
\underset{x\rightarrow\infty}{\sim}&\gamma\mathbb{E}[N(T)]\overline{F_{X}^{I}%
}(x)\mathbb{E}[h_{0}^{\alpha}(T,V_{0})]
\end{align*}
and the desired result follows.
\end{proof}

Now, we bear our interest on the Integrated Expected Severity, denoted by
IES(x) for any $x>0$, which deals with the average of the cumulated exceedences when the process is over the threshold $x$ on a time window $[0,T]$.

\begin{proposition}
Assume that the conditions of Corollary \ref{cor:var} hold. The Integrated Expected Severity for
a process defined as \eqref{SNP}\ for large values of $x$ is given by
\[
IES(x)\underset{x\rightarrow\infty}{\sim}\mathbb{\gamma}
{
\int\limits_{0}^{T}
}\mathbb{E}[N(t)]\mathbb{E}\left[  h_{0}^{\alpha}(t,V_{0})\right]  dt\,\overline{F_{X}^{I}}(x).
\]
\end{proposition}

\begin{proof}
By definition,
\[
IES(x)=\int_{0}^{T} \mathbb{E}[Y(t)-x]_{+} dt.
\]
From the previous Proposition, we directly obtain the result.
\end{proof}
\subsubsection{The Expected Time Over a Threshold}

The Expected Time Over a reference
threshold $x$, written \textit{ETOT(x)}, provides information about
how long does the process stay, in average, above a threshold $x$, knowing that it has already reached it. It is defined for all $x>0$ by
\[
ETOT(x)=\mathbb{E}\left[
\int\limits_{0}^{T}
\mathrm{1}\hspace{-0.35em}\mathrm{1}_{\{Y(t)\in]x,\infty\lbrack\}}dt\mid \max_{0\le t\le T}Y(t)>x\right].
\]

\begin{proposition}\label{etot}
Assume that the conditions of Corollary \ref{cor:var} hold. For large values of $x$, the ETOT(x) for the process defined as
\eqref{SNP} is given by
\begin{equation*}
ETOT(x)\underset{x\rightarrow\infty}{\sim}\frac{\int\limits_{0}^{T}
m(t)\mathbb{E}[h_{0}^{\alpha}(t,V_{0})]dt}{m(T)\mathbb{E}
[h_{0}^{\alpha}(V_{0},V_{0})]}.%
\end{equation*}
\end{proposition}
\begin{proof}
By definition,
\[
ETOT(x)=\frac{\mathbb{E}\left[
\int\limits_{0}^{T}
\mathrm{1}\hspace{-0.35em}\mathrm{1}_{\{Y(t)\in]x,\infty\lbrack\}}dt\right]}{\psi(x,T)}.
\]
Let $A(N(t))$ be the diagonal matrix such that for all $1\leq j\leq
N(t),~a_{jj}=h_{j}(t,T_{j})$. Note that%
\[
{\int\limits_{0}^{T}}
\mathbb{E}[\mathrm{1}\hspace{-0.35em}\mathrm{1}_{\{Y(t)>x\}}]dt=
{\int\limits_{0}^{T}}
\mathbb{P}({}Y\mathbf{(}t)>x)dt=
{\int\limits_{0}^{T}}
\mathbb{P}(\|C(N(t))\|_{1}>x)dt.
\]
Plug-in Corollary \ref{cor:var} in the previous expression concludes the proof.
\end{proof}
\begin{remark}[The extremal index for shot noise processes]
Under the previous assumptions, we define the extremal index $\theta\in[0,\infty]$ as the inverse of $\lim\limits_{T\to\infty}\lim\limits_{x\to\infty}ETOT(x)$ if it exists, \textit{i.e.}
\begin{equation*}
\frac{\int\limits_{0}^{T}
m(t)\mathbb{E}[h_{0}^{\alpha}(t,V_{0})]dt}{m( T)\mathbb{E}
[h_{0}^{\alpha}(V_{0},V_{0})]} \underset{T \rightarrow \infty}{\longrightarrow} \dfrac{1}{\theta}.
\end{equation*}
It can be seen as a continuous version of the extremal index for discrete time-series; see \cite{Rootzen}. To be precise, it measures the clustering tendency of high threshold exceedences and how the extreme values cluster
together. In this context, $\theta$ does not still belong to $[0,1]$ but to $[0, \infty]$. Notwithstanding, the inverse of the extremal index $\theta^{-1}$ indicates somehow, how long (in mean) an extremal event will occur, due to the dependency structure of the data. For instance $\theta=0$ for a random walk such that $h_i=1$ and extremal events may long forever. At the opposite, in the asymptotic independent case $h_i(t,v)=\1_{\left\lbrace t=v \right\rbrace}$, then $\theta=+\infty$ and extremal events occur instantaneous only.  
\end{remark}

\subsection{Application in dietary risk assessment and in non-life insurance mathematics}\label{KDEM}
The shot noise process defined as in \eqref{SNP} intervenes in many applications in which
sudden jumps occur such as in insurance to model the amount of aggregate claims that an insurer has to cope
with; see \cite{Assmussen 2010} and \cite{Embr_1997} and \cite{Mikosch}.
 
In dietary risk assessment and non-life insurance, we typically consider deterministic shock
functions defined for all $0\leq x\leq t$ by $h(t,x)=e^{-\omega
(t-x)}$, with a shape parameter $\omega>0$; see \cite{Bert_2008}, \cite{Bert_2010} and \cite{Mikosch} for more
details. We call this model \textit{Exp-SNP} for Exponential Shot Noise Process.

In \cite{Bert_2008}, authors suggested a model, called Kinetic Dietary Exposure Model (\textit{KDEM}), to represent the evolution of a contaminant in the human body. Their model is a discrete-time risk process which can be expressed from a \textit{Exp-SNP} on the shock intants; see Remark \ref{rem:KDEM}. In this context, shocks are regularly varying distributed intakes which arise according to $N(T)$ and $\omega$ is an elimination parameter.

We consider below an extension of the KDEM process. The main novelty is as follows: we are interested in the case when the $\omega=\omega_i\in\Omega$ are \textit{i.i.d. r.v.}'s and may take negative values satisfying the Cramer's condition $\E[\exp(p\omega_-T)]<\infty$ for some $p>\alpha$ and $\omega_-=\max(-\omega,0)$. Then, the model is defined by 
\begin{equation}\label{ExtendedSNP}
Y(t)=\sum\limits_{i=1}^{N(t)}X_{i}e^{-\omega_i(t-T_i)},\qquad\forall\ t\geq0.
\end{equation}
Here again, we can assume the the $(X_i)_{i\geq 1}$ are asymptotically independent and identically regularly varying random variables. In dietary risk assessment, it makes sense to consider such random elimination parameter to take into account interactions between different human organs. Thus, $\omega$ can be seen as a "inhibitor factor", for positive values of $\omega$, or contrariwise, a "catalytic factor" for negative values of $\omega$.
In insurance, these \textit{Exp-SNP} are often used and the parameter $\omega_i=\omega$ can be seen as an accumulation (resp. discount) factor when $\omega$ is a strictly negative (resp. positive) constant.
Here the model is such that usually a discount
factor applies on the risk except in some cases when it is the opposite and there
is accumulation of the risk. 
  
Assuming an homogeneous Poisson process on
the distribution of claims instants such that $\lambda(s)=\lambda$ for all
$0<s\leq t$, we obtain explicit formula for all the risk measures. Precisely,
\begin{corollary}
Let $Y(t)$ follow the KDEM defined in \eqref{ExtendedSNP} with random elimination parameter $\omega$ such that \textbf{(H0')-(H1)-(H2)} hold. Then we have $\theta = \alpha \dfrac{\mathbb{E}[ \omega^2]}{\mathbb{E}[\omega]}$ for $\omega>0$ a.s., $\theta = \alpha \omega_-$ if $\P(\omega<0)>0$ and $\omega_-$ constant, and for each risk indicator
\begin{align*}
\mathbb{P}\left(  Y(T)>x\right) & \underset{x\rightarrow\infty}{\sim}
\lambda\mathbb{E}_{\Omega}\left[  \frac{(1-e^{-\alpha\omega T})}{\omega\alpha
}\right]  \overline{F}_{X}(x), \\
\psi(x,T)&\underset{x\rightarrow\infty}{\sim}\lambda
\left(T\,  \mathbb{P(}\omega>0)+\mathbb{E}\left[  \frac{1-e^{-\alpha\omega T}}{{\omega\alpha}}\mathrm{1}\hspace{-0.35em}\mathrm{1}
_{\{\omega\leq0\}}\right]  \right)  \overline{F}_{X}(x), \\
ETOT(x)&\underset{x\rightarrow\infty}{\sim}\frac{\mathbb{E}_{\Omega}\left[
\frac{\left(  \omega\alpha T+e^{-\omega T\alpha}-1\right)  }{\left(
\omega\alpha \right)  ^{2}}\right]}{ T\,  \mathbb{P(}\omega>0)+\mathbb{E}\left[  \frac{1-e^{-\alpha\omega T}}{{\omega\alpha}}\mathrm{1}\hspace{-0.35em}\mathrm{1}
_{\{\omega\leq0\}}\right]   },\\
IES(x)&\underset{x\rightarrow\infty}{\sim}\lambda\,\mathbb{\gamma\, E}_{\Omega
}\left[  \frac{\left(  \omega\alpha T+e^{-\omega T\alpha}-1\right)  }{\left(
\omega\alpha\right)  ^{2}}\right]  \overline{F_{X}^{I}}(x).
\end{align*}
\end{corollary}

Note that we obtained an explicit and understandable formulae for the extremal index. In dietary risk assessment, the elimination parameter $\omega$ is most of the time chosen as a positive constant. Then, when $x$ is large, $\theta = \alpha \omega $ and one can estimate the tail index $\alpha$ with usual  statistical methods like the Hill estimator. As mentioned before, it could make sense that the elimination parameter may vary beeing random and may take negative values. Nonetheless, estimate its first and second moments seems to be a more difficult task and is left for further investigations.

\begin{proof}
Note that the only difficulty is the computation of the ruin probability.
Indeed, it is enough to take the expectation regarding $\omega$ and to apply
formula given above to get the others risk indicators.

Remind that the ruin probability deals with the maxima of $S$ which necessarily arise either on the skeleton, \textit{i.e.} on the claims instants $T_i$'s or between $T$ and
the last intake instant $T_{N(T)}$. Then, we have
\[
\psi(x,T)=\max\left\lbrace  \underset{1\leq k\leq N(T)}{\max}
\sum\limits_{i=1}^{k}
e^{-\omega_{i}(T_{k}-T_{i})}X_{i},\sum\limits_{i=1}^{N(T)}
e^{-\omega_{i}(T-T_{i})}X_{i}\ \right\rbrace .
\]
Note that,%
\[
{\sum\limits_{k=1}^{N(T)}}||A_{k}(N(T))||_{\infty}^{\alpha}={\sum\limits_{k=1}^{N(T)}%
}\underset{j}{{\max}}\left\{  e^{-\omega_{k}(T-T_{k})},e^{-\omega_{k}
(T_{k}-T_{j})}\right\}  ^{\alpha}.
\]
We obtain 
\begin{align*}
\mathbb{E}\left[  {\sum\limits_{k=1}^{N(T)}}||A_{k}(N(T))||_{\infty}^{\alpha
}\right]    & =\mathbb{E}\left[  {\sum\limits_{k=1}^{N(T)}}\left(  \mathrm{1}%
\hspace{-0.35em}\mathrm{1}_{\{\omega_{k}>0\}}+e^{-\omega_{k}(T-T_{k}
)}\mathrm{1}\hspace{-0.35em}\mathrm{1}_{\{\omega1\leq0\}}\right)  ^{\alpha
}\right]  \\
& =\mathbb{E}\left[  {\sum\limits_{k=1}^{N(T)}}\mathbb{P(}\omega_{1}
>0)+\mathbb{E}\left[  e^{-\alpha\omega_{1}(T-V_{0})}\mathrm{1}\hspace
{-0.35em}\mathrm{1}_{\{\omega_{1}\leq0\}}\right]  \right]  \\
& =\mathbb{E}\left[  N(T)\right]  \left(  \mathbb{P(}\omega_{1}>0)+\mathbb{E}%
\left[  e^{-\alpha\omega_{1}T}\left(
\int\limits_{0}^{T}
e^{-\alpha\omega_{1}t}dt\right)  \mathrm{1}\hspace{-0.35em}\mathrm{1}
_{\{\omega_{1}\leq0\}}\right]  \right)  \\
& =\mathbb{E}\left[  N(T)\right]  \left(  \mathbb{P(}\omega_{1}>0)+\mathbb{E}
\left[  \frac{1-e^{-\alpha\omega_{1}T}}{{\alpha\omega_{1}T}}\mathrm{1}
\hspace{-0.35em}\mathrm{1}_{\{\omega_{1}\leq0\}}\right]  \right)  .
\end{align*}
The desired result follows.
\end{proof}

Note that we can also deal with strictly concave functions $h_{j}(\cdot,T)$ by considering the new arrival times corresponding to the delayed maxima. However, the associated counting process is no longer a Poisson process and the constants are in general less explicit.

\begin{remark}\label{rem:KDEM}
Note that for the Exp-SNP with constant $\omega$, the asymptotic ruin equivalent $\psi^\ast(x,T)\sim\lambda T  \overline{F}_{X}(x)$ does not depend on the distribution of $\omega$. We have
\[
\lim_{x\to\infty}\frac{\psi(x,T)}{\psi^\ast(x,T)}=  \mathbb{P}(\omega>0)+\mathbb{E}\left[  \frac{1-e^{-\alpha\omega T}}{T\omega\alpha}\mathrm{1}\hspace{-0.35em}\mathrm{1}%
_{\{\omega\leq0\}}\right]\ge \frac{1+\E[e^{-\alpha\omega_- T}]}2.
\]
Notice that the original KDEM with r.v. $\omega$ has been defined in \cite{Bert_2008} by the recursive equation
\[
Y_{T_{j+1}}=\exp(-\omega_{j+1}\Delta T_{j+1})Y_{T_j}+X_{T_{j+1}},\qquad j\ge0,
\]
with $\Delta T_{j+1}=T_{j+1}-T_j$. This model is equivalent to KDEM with rate $1$ and with inter-arrivals $\omega_{j+1}\Delta T_{j+1}$. This process converges to a stationary solution under $\E[\omega]>0$ that is assumed from now, see \cite{Bougerol}. Applying Remark \ref{rem:ruin} we obtain the ruin probability for that model
\[
\lim_{x\to\infty}\frac{\tilde \psi(x,t)}{\P(X_1>x)}=\lambda T\E[\omega].
\]
\end{remark}

\begin{remark}
Let us denote for all $n \in \mathbb{N}$, $Y_{T_{n+1}}=Y_{n+1}$, the chain on jump instants. Then, the embedded chain of the KDEM process with a constant elimination parameter is defined by 
\begin{equation*}
Y_{n+1}=e^{-\omega \Delta T_{n+1}}Y_n+X_{n+1}.
\end{equation*}
Then, thanks to \cite{perfekt}, it follows that 
\begin{align*}
\theta=1-\mathbb{E}[e^{-\alpha \omega\Delta_T}].
\end{align*} 
If $\Delta_T$ is exponentially distributed with rate $\lambda$, 
\begin{align*}
\mathbb{E}[e^{-\alpha \omega\Delta_T}]=\dfrac{\lambda}{\lambda+\alpha\omega},
\end{align*} 
and we have 
\begin{align*}
\theta=\dfrac{\alpha\omega}{\lambda+\alpha\omega}.
\end{align*} 
Remark that the result differs from ours. A coefficient $1/(\lambda+\alpha\omega)$ appears and no interpretation or comparison between the extremal index for discrete-time series and its continuous equivalent is possible. However, the inverse of extremal index for discrete-time series gives, in average, the number of extremes by cluster. Then, the cluster are rougly of size $(1+\lambda^{-1}\alpha\omega)/\alpha\omega$. When $\lambda \rightarrow \infty$, which corresponds to expand time to move from discrete to continuous setting, it converges to $ \alpha\omega$ and we recover our result regarding the continuous version.  
\end{remark}

To conclude, in many configurations, under \textbf{(H0')}, we can explicitly derive the
constant $\sum_{k=1}^N \Vert A_{k}(N)\Vert^{\alpha}$, especially with respect to 
$\Vert\cdot\Vert_{1}$ and $\Vert\cdot\Vert_{\infty}$ which provide interesting equivalents to
obtain risk indicators. We used it to compute the tail process, the ruin
probability, the ETOT and the IES but our result can be applied on many other
risk measures like Gerber-Shiu measures. In this paper, we have proposed to focus on \textit{SNP} because it plays a leading role in risk
theory but note that modifying the matrix $A(N)$,
our method can be applied on several others stochastic processes.
Finally, note that we can extend the previous results when $N$ is not a Poisson process but admits some finite moments so that Lemma \ref{lem:h4} holds. 

\section{Proofs of the main results}

\subsection{Preliminaries}
We begin by providing some useful properties to prove Propositions \ref{regvect} and \ref{Caracspecmea} and Theorems \ref{Breiman} and \ref{th:regvar}.

\begin{remark}\label{relation norm} 
Results presented throughout the paper remain valid for any norm $\Vert\cdot \Vert$ such that
\textbf{(H3) }holds. For any $n\in \mathbb{N}$, and $x>0$, we have 
\begin{equation*}
\mathbb{P(}\Vert X(n)\Vert_{\infty}>x)\leq\mathbb{P(}\Vert X(n)\Vert >x)\leq\mathbb{P(}%
\Vert X(n)\Vert_{1}>x). 
\end{equation*}
\end{remark}

We will use several times the following result known as Potter's bound.
\begin{proposition}\label{Potter}
Let $\overline{F} \in RV_{- \alpha}$. Under \textbf{(H0)}, there exists $\epsilon>0$ such that $\mathbb{E}[N^{\alpha+\epsilon+1}]<\infty$ and there exist $x_0>0$, $c>1$ such that for all $y \geq x \geq x_0$, we have   
\begin{equation*}
 c^{-1}(y/x)^{-\alpha - \epsilon} \leq \dfrac{\overline{F}(y)}{\overline{F}(x)}\leq c(y/x)^{-\alpha + \epsilon}.
\end{equation*}
\end{proposition}

\begin{proof}
See \cite{Soulier}.
\end{proof}
Let us provide a technical lemma useful in the proofs:
\begin{lemma}\label{bornasymp}
Let $X(n)=(X_{1},X_{2},\ldots,X_{n})$ be a sequence such that
\textbf{(H1)-(H2) }hold. Let $f$ and $g$ be strictly positive functions such
that for all $n\in \mathbb{N}
^{\ast}$, $f(n)\leq g(n)\leq1$. Then, for any fixed $n \geq 1$ and $\epsilon>0$, there exists $b(x) \underset{x \rightarrow \infty}{\rightarrow}0$ such that
\begin{equation*}
\sum\limits_{i=1,i \neq j}^n \frac{\mathbb{P}\left(  X_{i}>f(n)x\ ,X_{j}>g(n)x\right)  }{\mathbb{P}\left(
X_{j}>x\right)  }\leq f(n)^{-\alpha+\epsilon} n^2 b(x). 
\end{equation*}
\end{lemma}
\begin{proof}
Let $1\leq i\neq j\leq n$, with $j\geq 1$ and $n\in \mathbb{N}
^{\ast}$. Then, using Potter's bound, for $x$ sufficiently large
\begin{align*}
\sum\limits_{i=1,i \neq j}^n \frac{\mathbb{P}\left(  X_{i}>f(n)x\ ,X_{j}>g(n)x\right)  }{\mathbb{P}\left(
X_{j}>x\right)  }  &  \leq \sum\limits_{i=1 \neq j}^n \frac{\mathbb{P}\left(  X_{i}>f(n)x\ ,X_{j}%
>f(n)x\right)  }{\mathbb{P}\left(  X_{j}>x\right)  }\\
&  \leq \sum\limits_{i=1,i \neq j}^n\frac{\mathbb{P}\left(  X_{i}>f(n)x\ ,X_{j}>f(n)x\right)
/\mathbb{P}\left(  X_{j}>f(n)x\right)  }{\mathbb{P}\left(  X_{j}>x\right)
/\mathbb{P}\left(  X_{j}>f(n)x\right)  }\\
&  \leq c \sum\limits_{i=1,i \neq j}^nf(n)^{-\alpha+\epsilon}\frac{\mathbb{P}\left(  X_{i}>f(n)x\ ,X_{j}%
>f(n)x\right)  }{\mathbb{P}\left(  X_{j}>f(n)x\right)  }\\
&  \leq c f(n)^{-\alpha+\epsilon} \sum\limits_{i=1,i \neq j}^n \underset{i,j}{\sup}\frac{\mathbb{P}\left(
X_{i}>f(n)x\ ,X_{j}>f(n)x\right)  }{\mathbb{P}\left(  X_{j}>f(n)x\right)  }\\
& \leq cf(n)^{-\alpha+\epsilon}n(n-1)\left(\underset{i,j}{\sup}\frac{\mathbb{P}\left(
X_{i}>f(n)x\ ,X_{j}>f(n)x\right)  }{\mathbb{P}\left(  X_{j}>f(n)x\right)  }\right),
\end{align*}
which combined with \textbf{(H2)} concludes the proof.
\end{proof}

The following lemma plays a leading role in the sequel. It can be seen as an uniform integrability condition which allows in the main body of Propositions \ref{regvect} and \ref{Caracspecmea} to integrate with respect to $N$; see \cite{Billingsley}, Section 3.
\begin{lemma} \label{lem:unifconv}
Let $N$ be a random length satisfying {\bf (H0)}.
Let  $X=(X_1,X_2,\ldots)\in \mathbb{R}_+^{\infty}$ be a sequence  such that \textbf{(H1)-(H2)} hold. Let $A=(a_{i,j})_{i,j \geq 1}$ be the double indexed sequence of the coefficients satisfying {\bf (H4)} and define $\mathbb{E}_A[\cdot]$ (respectively $\mathbb{P}_{X}[\cdot]$) the expectation (resp. the probability) with respect to $A$ (resp. $X$). Then, for any fixed $n_0 \in \mathbb{N}^*$ and for any norm $\Vert \cdot \Vert$ such that \textbf{(H3)} holds, we have the following statement:
$$
\underset{n_0\rightarrow\infty}{\lim} \sup\limits_{x>0} \left( \mathbb{E}_{A} \left[\sum\limits_{n=n_0+1}^{\infty} \mathbb{P}(N=n) \dfrac{\mathbb{P}(\Vert A(n)  X(n) \Vert >x)}{\mathbb{P}(X_1>x)}\right]\right)=0.
$$
\end{lemma}

\begin{proof}
Note first that 
\begin{align*}
I(x) & = \sup\limits_{x>0} \left(\mathbb{E}_{A} \left[ \sum\limits_{n=n_0+1}^{\infty} \mathbb{P}(N=n) \dfrac{\mathbb{P}(\Vert A(n) \Vert \Vert X(n) \Vert >x)}{\mathbb{P}(X_1>x)}\right]\right)\\ 
&  \leq \mathbb{E}_{A} \left[ \sup\limits_{x>0} \left( \sum\limits_{n=n_0+1}^{\infty} \mathbb{P}(N=n) \dfrac{\mathbb{P}((\Vert A(n) \Vert \vee 1) \Vert X(n) \Vert >x)}{\mathbb{P}(X_1>x)} \right)\right]\\
\end{align*}
where $(a \vee b)=\max(a,b)$ for any $a,b \in \mathbb{R}^+$.
For a fixed $x_0>0$ defined as in Proposition \ref{Potter}, denoting $y=x_0(\Vert A(n) \Vert \vee 1)$, we have
\begin{align*}
I(x) & \leq  \mathbb{E}_A \left[ \sup\limits_{x>0} \left( \sum\limits_{n=n_0+1}^{(x/y \vee n_0)} \mathbb{P}(N=n)\dfrac{\mathbb{P}((\Vert A(n) \Vert \vee 1)\Vert X(n) \Vert >x)}{\mathbb{P}(X_1>x)} \right) \right]\\  
&   \qquad \qquad \qquad \qquad  + \mathbb{E}_A \left[ \sup\limits_{x>0} \left( \sum\limits_{n>(x/y \vee n_0)}^{\infty} \mathbb{P}(N=n) \dfrac{\mathbb{P}((\Vert A(n) \Vert \vee 1) \Vert X(n) \Vert >x)}{\mathbb{P}(X_1>x)}\right)\right]\\ 
&  \leq  \mathbb{E}_{A} \left[ \sup\limits_{x>ny} \left( \sum\limits_{n=n_0+1}^{\infty} \mathbb{P}(N=n) \dfrac{\mathbb{P}((\Vert A(n) \Vert \vee 1) \Vert X(n) \Vert >x)}{\mathbb{P}(X_1>x)} \right)\right] \\
&  \qquad \qquad \qquad \qquad  + \mathbb{E}_{A} \left[ \sup\limits_{x>0} \left(\sum\limits_{n>(x/y \vee n_0)}^{\infty} \mathbb{P}(N=n) \dfrac{\mathbb{P}((\Vert A(n) \Vert \vee 1) \Vert X(n) \Vert >x)}{\mathbb{P}(X_1>x)}\right)\right]\\
&  \leq \mathbb{E}_A \left[ \sup\limits_{x>ny} \left(\sum\limits_{n=n_0+1}^{\infty} \mathbb{P}(N=n) \dfrac{\mathbb{P}((\Vert A(n) \Vert \vee 1) \Vert X(n) \Vert >x)}{\mathbb{P}(X_1>x)}\right)\right]\\
& \qquad \qquad \qquad \qquad +  \mathbb{E}_A \left[ \sup\limits_{x>0} \left( \dfrac{\mathbb{P}(N>(x/y \vee n_0)}{\mathbb{P}(X_1>x)}\right)\right] \\
& \leq I_1(x)+I_2(x).
\end{align*}
Let us first investigate $I_1(x)$. For every fixed $x>0$ and $n>0$, we have 
\begin{align*}
\mathbb{E}_{A}\left[ \mathbb{P}\left((\Vert A(n) \Vert \vee 1) \Vert X(n)\Vert >x \right)\right] &  \leq
\mathbb{E}_{A}\left[ \mathbb{P}\left( \Vert X(n)\Vert_1 >x/(\Vert A(n) \Vert \vee 1) \right)\right]\\
&  \leq \mathbb{E}_{A}\left[ \sum\limits_{i=1}^n \mathbb{P}_{X_1}\left(  X_i>x/n(\Vert A(n) \Vert \vee 1) \right)\right]\\
&  = n\mathbb{E}_{A}\left[ \mathbb{P}_{X_1}(X_1>x/n(\Vert A(n) \Vert \vee 1))\right]. 
\end{align*}
Then, 
$$I_1(x) \leq \mathbb{E}_{ A } \left[ \sup\limits_{x>ny} \left(  \sum\limits_{n=n_0+1}^{\infty}
n\mathbb{P}(N=n) \dfrac{\mathbb{P}_{X_1}(X_1>x/n(\Vert A(n) \Vert \vee 1))}{\mathbb{P}(X_1>x)}\right) \right].$$
Using the Potter's bound, there exists $c>1$ independent of $y\ge 1$ such that
$$ \sup_{x \geq ny} \dfrac{\mathbb{P}_{X_1}(X_1>x/n(\Vert A(n) \Vert \vee 1))}{\mathbb{P}(X_1>x)} \leq cn^{\alpha+\epsilon}(\Vert A(n) \Vert^{\alpha+\epsilon} \vee 1).$$
It follows from \textbf{(H4)} that 
$$ I_1(x) \leq c \sum\limits_{n=n_0+1}^{\infty}
\mathbb{P}(N=n) n^{\alpha+1+\epsilon}\mathbb{E} \left[(\Vert A(N)\Vert ^{\alpha+\epsilon} \vee 1) | N=n \right] \underset{n_0\rightarrow\infty}{\longrightarrow} 0.$$
We focus now on $I_2(x)$. From Markov inequality, for any $x>0$ and for any $\epsilon>0$, under \textbf{(H0)}, there exists a constant $c>0$ such that 
\begin{align*}
\mathbb{E}_A[\mathbb{P}(N>(x/y \vee n_0))] \leq \mathbb{E}_A\left[\dfrac{\mathbb{E}[N^{\alpha+\epsilon }]}{(x/y \vee n_0)^{\alpha+\epsilon }}\right]\leq c\, \mathbb{E}_A\left[\dfrac{1}{(x/y \vee n_0)^{\alpha+\epsilon }}\right].
\end{align*}
Moreover,  we have
\begin{align*} 
\mathbb{E}_A \left[\sup\limits_{x \leq n_0y}\dfrac{\mathbb{P}(N>(x/y \vee n_0))}{\mathbb{P}(X_1>x)}\right] &\leq c \left( \mathbb{E}_A \left[ \sup\limits_{x \leq n_0y} \dfrac{(x/y \vee n_0)^{-\alpha-\epsilon  }}{\mathbb{P}(X_1>x)} \right] +\mathbb{E}_A \left[ \sup\limits_{x \geq n_0y} \dfrac{(x/y \vee n_0)^{-\alpha-\epsilon  }}{\mathbb{P}(X_1>x)}\right] \right)\\
&\leq c \left( \mathbb{E}_A \left[   \dfrac{ n_0^{-\alpha-\epsilon  }}{\mathbb{P}(X_1>n_0y)} \right] +\mathbb{E}_A \left[ \sup\limits_{x \geq n_0y} \dfrac{(y/x)^{\alpha+\epsilon  }}{\mathbb{P}(X_1>x )}\right] \right)
\end{align*}
Notice that  that from {\bf (H4)} the moments of order $\alpha+\epsilon$ of $\|A\|$ are finite and so $\E[y^{\alpha+\epsilon}]<\infty$. 
We use again the Potter's bound and, for a possibly different constant $c>0$ (independent of $y>1$) and $n_0$ sufficiently large, we obtain
\[
I_2(x)  \leq c\mathbb{E}_A \left[  y^{\alpha+\epsilon} \right] \left( \dfrac{ n_0^{-\alpha-\epsilon }}{\mathbb{P}(X_1>n_0 )}  + \sup\limits_{x \geq n_0 } \dfrac{x^{-\alpha-\epsilon }}{\mathbb{P}(X_1>x )} \right) \underset{n_0\rightarrow\infty}{\longrightarrow} 0.
\]
We finally obtain 
$$\lim\limits_{n_0 \rightarrow \infty} I(x)=0$$
which concludes the proof.
\end{proof}

\subsection{Proof of Proposition \ref{regvect}}

We first state a lemma which can be seen as a generalization of a well-known property for
\textit{i.i.d.} regularly varying \textit{r.v.}'s with respect to the infinite norm
$\Vert\cdot \Vert_{\infty}$, which means that the maximum of a sequence
satisfying \textbf{(H1)} is reached just by one coordinate. Note the crucial role of the uniform
asymptotic independence condition \textbf{(H2)} in the sequel.

\begin{lemma}\label{Conicreg} Let $X(n)=(X_{1},X_{2},\ldots,X_{n})$ be a sequence of r.v.'s
such that \textbf{(H1)}-\textbf{(H2)} hold and $\Vert\cdot\Vert$
satisfies \textbf{(H3)}. Then,
$$
\underset{x \rightarrow \infty}{\lim} \frac{\mathbb{P}(||X(n)||>x)}%
{n\mathbb{P}(X_{1}>x)}=1,
$$
for any fixed $n \geq 1$.
\end{lemma}

\begin{proof}
We proceed by upper and lower bounding the quantity
\[
A(x)=\frac{\mathbb{P}(||X(n)||>x)-n\mathbb{P}(X_{1}>x)}{n\mathbb{P}%
(X_{1}>x)}
.\]
From Remark \ref{relation norm}, under \textbf{(H3)}, it is enough to investigate the lower (respectively the upper) bound with respect to $\Vert \cdot\Vert_{\infty}$ (resp. $\Vert \cdot\Vert_1$).
For the lower bound, using the Bonferroni bound, for any fixed
$n\geq1$ and $x>0$, we have
\begin{equation*}
\mathbb{P}(\Vert X(n)\Vert_{\infty}>x)=\mathbb{P}\left(
\bigcup _{i=1}^{n}
X_{i}>x \right)  \geq n\mathbb{P}\left(  X_{i}>x \right)  -
\sum\limits_{i=1,i\neq j}^{n} \mathbb{P}( X_i>x , X_j>x ).
\end{equation*}
Then, from Lemma \ref{bornasymp} and Remark \ref{relation norm}, it follows that%
\[
A(x) \geq-
\sum\limits_{i=1,i\neq j}^{n}
\frac{\mathbb{P}(X_{i}>x,X_{j}>x)}{\mathbb{P}(X_{j}>x)}\underset{x\rightarrow\infty}{\longrightarrow} 0.
\]
For the upper bound, let us consider $\varepsilon$ such that $\frac{1}%
{2}<\varepsilon<1$. Then, for all $x>0$ and $n\geq1$,
\begin{align*}
\mathbb{P}(\Vert X(n)\Vert_{1} & >x)=\mathbb{P}\left(
\sum\limits_{i=1}^{n}
X_{i}>x\right)  \leq\mathbb{P}\left(
\bigcup _{i=1}^{n}
X_{i}>\varepsilon x\right)  +\mathbb{P}\left(
\sum\limits_{i=1}^{n}
X_{i}>x,\bigcap\limits_{j=1}^{n}
(X_{j}\leq\varepsilon x)\right) \\
&  =A_{1}(\varepsilon,x)+A_{2}(\varepsilon,x).
\end{align*}
Using an union bound again and \textbf{(H1)}, we obtain
\[
\underset{x\rightarrow\infty}{\limsup} \; \frac{A_{1}(\varepsilon,x)}{\sum\limits_{i=1}^{n}
\mathbb{P}\left(  X_{i}>x\right)  }\leq\underset{x\rightarrow\infty}{\ \limsup \;
}\frac{\sum\limits_{i=1}^{n}
\mathbb{P}\left(  X_{i}>\varepsilon x\right)  }{\sum\limits_{i=1}^{n}
\mathbb{P}\left(  X_{i}>x\right)  }=\varepsilon^{-\alpha}.
\]
Letting $\varepsilon\rightarrow1^{-}$, we obtain%
\[
\underset{x\rightarrow\infty}{\limsup} \; \frac{A_{1}(\varepsilon,x)}{\sum\limits_{i=1}^{n}
\mathbb{P}\left(  X_{i}>x\right)  }\leq1
\]
which implies that
\[
\underset{x\rightarrow\infty}{\limsup} \; \mathbb{P}\left(  \Vert X(n)\Vert_{1}>x\right)  -%
\sum\limits_{i=1}^{n}
\mathbb{P}\left(  X_{i}>x\right)  \leq\ \limsup_{\varepsilon\rightarrow1^{-}}\underset{x\rightarrow\infty}{\limsup
} \; A_{2}(\varepsilon,x).%
\]
On the other hand, we have
\begin{align*}
A_{2}(\varepsilon,x)  &  =\mathbb{P}\left(
\sum\limits_{i=1}^{n}
X_{i}>x,\
\bigcap\limits_{j=1}^{n}
(X_{j}\leq\varepsilon x),\underset{1\leq k\leq n}{\ \max}X_{k}>\frac{x}%
{n}\right) \\
&  \leq%
\sum\limits_{k=1}^{n}
\mathbb{P}\left(
\sum\limits_{i=1}^{n}
X_{i}>x,\ X_{k}\leq\varepsilon x,\ X_{k}>\frac{x}{n}\right) \\
&  \leq
\sum\limits_{k=1}^{n}
\mathbb{P}\left(
\sum\limits_{i=1,i\neq k}^{n}
X_{i}>(1-\varepsilon)x,\ X_{k}>\frac{x}{n}\right) \\
&  \leq%
\sum\limits_{k=1}^{n}
\sum\limits_{i=1,i\neq k}^{n}
\mathbb{P}\left(  X_{i}>\frac{(1-\varepsilon)x}{n-1},\ X_{k}>\frac{x}%
{n}\right).
\end{align*}
Besides, applying Lemma \ref{bornasymp} with $f(n)=\dfrac{1-\varepsilon}{n-1}$ and $g(n)=\dfrac{x}{n}$, we obtain%
\begin{align*}
\underset{x\rightarrow\infty}{\limsup}\; A(x) \leq
\underset{x\rightarrow\infty}{\limsup} \; \frac{A_{2}(\varepsilon,x)}{\sum\limits_{i=1}^{n}
\mathbb{P}\left(  X_{i}>x\right)  }  &  \leq\underset{x\rightarrow\infty}%
{\limsup}\; \frac{\sum\limits_{k=1}^{n}
\sum\limits_{i=1,i\neq k}^{n}
\mathbb{P}\left(  X_{i}>\frac{(1-\varepsilon)x}{n-1},\ X_{k}>\frac{x}%
{n}\right)  }{%
\sum\limits_{i=1}^{n}
\mathbb{P}\left(  X_{i}>x\right)  }\\
&  \leq\underset{x\rightarrow\infty}{\limsup} \; %
\sum\limits_{k=1}^{n}
\sum\limits_{i=1,i\neq k}^{n}
\frac{\mathbb{P}\left(  X_{i}>\frac{(1-\varepsilon)x}{n-1},\ X_{k}>\frac{x}%
{n}\right)  }{\mathbb{P}\left(  X_{k}>x\right)  }\\
&  =0.%
\end{align*}
Collecting the bounds and using a sandwich argument leads to 
\begin{align*}
\underset{x\rightarrow\infty}{\lim}A(x)=0
\end{align*}
for any fixed $n \in \mathbb{N}^*$, which concludes the proof.
\end{proof}

\begin{proof}[ \textbf{Proof of Proposition \ref{regvect}}]
From Lemma \ref{Conicreg}, for every fixed $n_0>0$, we have
\begin{equation*}\label{Convn0}
\sum\limits_{n=1}^{n_0}
\mathbb{P}(N=n) \dfrac{\mathbb{P}(\Vert X(n) \Vert >x)}{\mathbb{P}(X_1>x)}
\underset{x \rightarrow \infty}{\longrightarrow}\sum\limits_{n=1}^{n_0}n
\mathbb{P}(N=n).\\
\end{equation*}
Besides, under \textbf{(H0)-(H3)}, taking $ A(n)=I_n$, with $I_n$ the identity matrix of size $n \times n$, the assumptions of Lemma \ref{lem:unifconv} hold. Then, from Remark \ref{relation norm},
\begin{equation*} 
\underset{n_0 \rightarrow \infty}{\lim} \sup_{x>0} \sum\limits_{n=n_0+1}^{\infty}
\mathbb{P}(N=n)\dfrac{\mathbb{P}(\Vert X(n) \Vert >x)}{\mathbb{P}(X_1>x)} =0, \\ 
\end{equation*}
for any norm $\Vert \cdot \Vert$ such that \textbf{(H3)} holds which insures the uniform integrability with respect to $N$. Letting $n_0 \rightarrow \infty$ in the first expression concludes the proof.\end{proof}

\subsection{Proof of Proposition \ref{Caracspecmea}}
Let $X(n)=(X_1,\ldots,X_n)$ be a sequence such that \textbf{(H1)-(H2)} hold and let $\Vert \cdot \Vert$ satisfying \textbf{(H3)}. 
We need the following characterization.
\begin{equation}\label{normThetan}
\mathcal{L} \left(\dfrac{X(n)}{\Vert X(n)\Vert} \; | \;\Vert X(n)\Vert>x  \right) \underset{x\rightarrow\infty
}{\longrightarrow} \mathcal{L}(\Theta(n)), 
\end{equation}
with $\Theta(n)$ which does not depend on the choice of the norm $\Vert \cdot \Vert$. Specifically $\P(\Theta(n)=e_j)=n^{-1}$ which is consistent with Remark \ref{rem:Thetaiid}. We did not find a proper reference of this simple result and its proof follows: by asymptotical independence, the support of the spectral measure is concentrated on the canonical basis $\cup_{j=1}^n\{e_j\}$ and the measure is fully characterized by the probability $p_j=\P(\Theta(n)=e_j)=\P(Y_j(n)\ge 1)$ where $Y_j(n)=\Theta_j(n)Y(n)$, $1\le j\le n$. By definition of $\Theta(n)$, we also have $\P(Y_j(n)\ge 1)=\lim_{x\to\infty}\P(X_j(n)\ge x\mid  \|X(n)\|\ge x)=n^{-1}$ from Lemma \ref{Conicreg}.\\

Now, let $X(N)=(X_1,\ldots,X_N)$ be a a sequence such that \textbf{(H0-(H2)} hold. We first need to prove that for any norm $\vert\vert \cdot \vert\vert$ satisfying \textbf{(H3)}, 
\begin{equation*}\label{eq:conv}
\underset{x \rightarrow \infty}{\lim} \mathbb{P}\left( \left\vert\left\vert \frac{X(N)}{\vert\vert X(N)\vert\vert} - \Theta(N) \right\vert\right\vert >\epsilon \ | \ \vert\vert X(N) \vert\vert >x \right)=0.
\end{equation*}
From \eqref{normThetan}, the Skorohod's representation theorem and Lemma \ref{Conicreg}, it follows that for any fixed $n \in \mathbb{N}^*$, for any norm $\Vert \cdot \Vert$ such that \textbf{(H3)} holds and for any $\epsilon>0$, we have
\begin{align*}
\underset{x \rightarrow \infty}{\lim} \dfrac{\mathbb{P}\left( \left\vert\left\vert \frac{X(n)}{\vert\vert X(n)\vert\vert} - \Theta(n) \right\vert\right\vert >\epsilon \ , \ \vert\vert X(n) \vert\vert >x \right)}{n\mathbb{P}(X_1>x)}=0.
\end{align*}
Then,
\begin{align*}
\underset{x \rightarrow \infty}{\lim} \dfrac{\mathbb{P}\left( \left\vert\left\vert \frac{X(n)}{\vert\vert X(n)\vert\vert} - \Theta(n) \right\vert\right\vert >\epsilon \ , \ \vert\vert X(n) \vert\vert >x \right)}{\mathbb{P}(X_1>x)}=0.
\end{align*}
Moreover, 
\begin{align*}
\dfrac{\mathbb{P}\left( \left\vert\left\vert \frac{X(n)}{\vert\vert X(n)\vert\vert} - \Theta(n) \right\vert\right\vert >\epsilon \ , \ \vert\vert X(n) \vert\vert >x \right)}{\mathbb{P}(X_1>x)} \leq \dfrac{\mathbb{P}(\Vert X(n) \Vert>x)}{\mathbb{P}(X_1>x)},
\end{align*}
and the uniform integrability criteria holds, which leads to 
\begin{align*}
\underset{x \rightarrow \infty}{\lim}\sum\limits_{n=1}^{\infty} \dfrac{\mathbb{P}\left( \left\vert\left\vert \frac{X(n)}{\vert\vert X(n)\vert\vert} - \Theta(n) \right\vert\right\vert >\epsilon \ , \ \vert\vert X(n) \vert\vert >x \right)}{\mathbb{P}(X_1>x)} \mathbb{P}(N=n)=0.
\end{align*}
Finally, thanks to   Proposition \ref{regvect} we know that $\Vert X(N) \Vert$ is regularly varying, and for any $\epsilon>0$ we have 
\begin{align*}
\mathbb{P}\left( \left\vert\left\vert \frac{X(N)}{\vert\vert X(N)\vert\vert} - \Theta(N) \right\vert\right\vert >\epsilon \ | \ \vert\vert X(N) \vert\vert >x \right)\underset{x \rightarrow \infty}{\longrightarrow}0.
\end{align*}

We can now proceed to the characterization of the spectral measure. For all $X(N)=(X_1,X_2,\cdots,X_N) \in c_{\Vert \cdot \Vert}\backslash\{{\bf 0}\}$, we have
\begin{align*}
\mathbb{P}\left( \Vert X(N)\Vert^{-1}X(N) \in \cdot\mid\Vert X(N)\Vert
>x\right) 
 & =\frac{\sum\limits_{n=1}^{\infty}\mathbb{P}\left(  \Vert
X(N)\Vert^{-1}X(N)\in\cdot\ ,\Vert X(N)\Vert>x\mid N=n\right)  \mathbb{P}%
(N=n)}{\sum\limits_{n=1}^{\infty}\mathbb{P}\left(  \Vert
X(N)\Vert>x\mid N=n\right)  \mathbb{P}(N=n)}\\
&  =\frac{\sum\limits_{n=1}^{\infty}\mathbb{P}\left(
\Vert X(n)\Vert^{-1}X(n)\in\cdot\ ,\Vert X(n)\Vert>x\right)  \mathbb{P}%
(N=n)}{\sum\limits_{n=1}^{\infty}\mathbb{P}\left(  \Vert
X(n)\Vert>x\right)  \mathbb{P}(N=n)}\\
&  =\frac{\sum\limits_{n=1}^{\infty}\mathbb{P}\left(  \Vert
X(n)\Vert^{-1}X(n)\in\cdot\ ,\Vert X(n)\Vert>x\right)  \mathbb{P}%
(N=n)\ /\ \mathbb{P}(X_{1}>x)}{\sum\limits_{n=1}^{\infty}\mathbb{P}\left(  \Vert X(n)\Vert>x\right)  \mathbb{P}(N=n)\ /\ \mathbb{P}%
(X_{1}>x)}.
\end{align*}
Then, from what preceeds and Lemma \ref{Conicreg}, we get for all $j\geq1$

\begin{align*}
\mathbb{P}\left( \Vert X(N)\Vert^{-1}X(N)= e_j \mid\Vert X(N)\Vert
>x\right) \underset{x\rightarrow\infty}{\sim}\frac{\sum\limits_{n=j}%
^{\infty}\mathbb{P}(N=n)}{\sum\limits_{n=1}^{\infty}%
n\mathbb{P}(N=n)}
  \underset{x\rightarrow\infty}{\sim}\frac{\mathbb{P}(N\geq j)}%
{\mathbb{E}[N]},
\end{align*}
and the desired result follows.

\subsection{Proof of Theorem \ref{Breiman}}

From Proposition \ref{cor:convprocess}, for any fixed $n_0$, we have
\begin{align*}
 \sum\limits_{n=1}^{n_0}  \dfrac{\mathbb{P}(\|C(n)\|>x )}{ \P(X_1>x)}\mathbb{P}(N=n) \underset
{x\rightarrow\infty}{\longrightarrow} \sum\limits_{n=1}^{n_0} \E\Big[\sum_{k=1}^n\|A_k(n)\|^\alpha\Big]\mathbb{P}(N=n).
\end{align*}
From Lemma \ref{lem:unifconv}, the uniform integrability of $\mathbb{P}(\Vert C(n) \Vert>x)/ \mathbb{P}(X_1>x)$ with respect to $N$ holds, one can let $n_0$ tends to $+\infty$ above which concludes the proof.

\subsection{Proof of Theorem \ref{th:regvar}}
Let us use a similar but slightly but more evolved reasoning to prove Theorem \ref{th:regvar}.
From Theorem \ref{Breiman},
if we assume \textbf{(H5)}, as the $a_{i,j}$ are not all identically null we have
\begin{equation}\label{extbreimanstrictpos}
\lim_{x\rightarrow\infty}\frac{\mathbb{P}(\Vert C(N)\Vert>x)}{\mathbb{P}%
(X_1>x)}=\mathbb{E}\left[
\sum_{j=1}^{N}\Vert A_{j}(N)\Vert^{\alpha}\right]  >0,
\end{equation}
and then $\Vert C(N)\Vert$ is regularly varying. It remains to prove the existence of the spectral measure. From Proposition \ref{regvect}, we have
\begin{align*}
P_{x,n}(\cdot)  &  =\mathbb{P}(\Vert C(N)\Vert^{-1}C(N)\in\cdot\ |\ \Vert C(N)\Vert
>x,N=n)\\
&  =\frac{\mathbb{P}(\Vert C(N)\Vert^{-1}C(N)\in\cdot,\Vert C(N)\Vert
>x,N=n)}{\mathbb{P}(\Vert C(N)\Vert>x,N=n)}\\
&  = \frac{\mathbb{P}(\Vert C(n)\Vert^{-1}%
C(n)\in\cdot,\Vert C(n)\Vert>x )}{\mathbb{P}%
(X_1>x)}\frac{\mathbb{P}(X_1>x)}{\mathbb{P}(\Vert C(n)\Vert>x)}.
\end{align*}
Consider $n$ sufficiently large such that $\mathbb{E}[\Vert
A(n)\Theta(n)\Vert^{\alpha}]>0$. Applying the regular varying properties stated in Proposition \ref{cor:convprocess},   we have%
\[
\lim_{x\rightarrow\infty}\frac{\mathbb{P}(X_1> x)}{\mathbb{P}%
(\Vert C(n)\Vert>x)}=\frac{1}{\mathbb{E}[\Vert
A(n)\Theta(n)\Vert^{\alpha}n]}.
\]
Then,
\[
\lim_{x\to\infty}P_{x,n}(\cdot)=\lim_{x\to\infty} \frac{  \mathbb{P}(\Vert
A(n)X(n)\Vert^{-1}A(n)X(n)\in\cdot,\Vert A(n)X(n)\Vert>x )}{\mathbb{E}[\Vert A(n)\Theta(n)\Vert^{\alpha}n]\,\mathbb{P}(X_1
> x)}%
\]
and, as $X(n)$ is regularly varying we can apply the multivariate Breiman's lemma to obtain
\[
\lim_{x\to\infty}P_{x,n}(\cdot)=\frac{\mathbb{E}[\Vert A(n)\Theta(n)\Vert^{\alpha}n\mathrm{1}\hspace
{-0.35em}\mathrm{1}_{\Vert A(n)\Theta(n)\Vert^{-1}A(n)\Theta(n)\in\cdot}]}%
{\mathbb{E}[\Vert A(n)\Theta(n)\Vert^{\alpha}n]}=\frac{\mathbb{E}\left[  \sum_{k=1}^{n}\Vert A_{k}(n)\Vert^{\alpha
}\mathrm{1}\hspace{-0.35em}\mathrm{1}_{\Vert A_{k}(n)\Vert^{-1}A_{k}%
(n)\in\cdot}\right]  }{\mathbb{E}\left[  \sum_{k=1}^{n}\Vert A_{k}%
(N)\Vert^{\alpha}\right]  }.%
\]
Now, for any norm $\Vert \cdot \Vert$ such that \textbf{(H3)} holds, we have 
\begin{align*}
\mathbb{P}\left(  \Vert C(n)\Vert^{-1}C(n)\in\cdot\mid\Vert C(n)\Vert
>x\right) \leq \dfrac{\mathbb{P}(\Vert C(n)\Vert_1>x)}{\mathbb{P}(X_1>x)}
\end{align*}
and the uniform integrability condition of Lemma \ref{lem:unifconv} is fulfilled which, combined with \eqref{extbreimanstrictpos}, leads to
\begin{align*}
\mathbb{P}\left(  \Vert C(N)\Vert^{-1}C(N)\in\cdot\mid\Vert C(N)\Vert
>x\right) \underset{x\rightarrow\infty
}{\longrightarrow} \frac{\mathbb{E}\left[  \sum_{k=1}^{N}\Vert A_{k}(N)\Vert^{\alpha
}\mathrm{1}\hspace{-0.35em}\mathrm{1}_{\Vert A_{k}(N)\Vert^{-1}A_{k}%
(N)\in\cdot}\right]  }{\mathbb{E}\left[  \sum_{k=1}^{N}\Vert A_{k}%
(N)\Vert^{\alpha}\right]  }.
\end{align*}

\textbf{Acknowledgment}. Charles Tillier would like to thank the Institute of Mathematical Sciences of the University of Copenhagen for hospitality and financial support when visiting Olivier Wintenberger. Financial supports by the ANR network AMERISKA  ANR 14 CE20 0006 01 are also gratefully acknowledged by both authors.

\end{document}